\pdfoutput=1
\documentclass[11pt]{article}
\usepackage{authblk}
\usepackage[toc,page]{appendix}
\usepackage[top=2cm, bottom=2cm, left=2cm, right=2cm]{geometry}

\usepackage{color}
\usepackage{helvet}         
\usepackage{courier}        
\usepackage{type1cm}        

\usepackage{framed}
\usepackage{tikz}
\usepackage{makeidx}         
\usepackage{graphicx}        
\usepackage{multicol}        
\usepackage[bottom]{footmisc}

\usepackage{amsmath}
\usepackage{amssymb}
\usepackage{bbold}
\usepackage{amsthm}
\usepackage{subcaption}
\usepackage{sidecap}
\usepackage{floatrow}
\usepackage{pdflscape}
\usepackage{pdflscape}
\usepackage{comment}
\usepackage[font=small]{caption}
\usepackage{enumitem}
\usepackage{esint}

\usepackage{scalerel}

\newtheorem{theorem}{Theorem}
\newtheorem{corollary}[theorem]{Corollary}
\newtheorem{lemma}[theorem]{Lemma}

\newtheorem{proposition}[theorem]{Proposition}
\newtheorem{proposition*}{Proposition}
\newtheorem{lemma*}{Lemma}

\theoremstyle{remark}

\newtheorem{remark}[theorem]{\bf Remark}
\newtheorem{definition}[theorem]{\bf Definition}

\numberwithin{theorem}{section}
\numberwithin{question}{section}
\numberwithin{figure}{section}
\numberwithin{equation}{section}

\begin{document}

\title{Decomposition of global 2-SLE for $\kappa\in (4,8)$\\ and an application to critical FK-Ising model}
\bigskip{}
\author[1]{Yu Feng\thanks{yufeng\_proba@163.com}}
\author[2]{Mingchang Liu\thanks{liumc\_prob@163.com}}
\author[1]{Hao Wu\thanks{hao.wu.proba@gmail.com.}}
\affil[1]{Tsinghua University, Beijing, China}
\affil[2]{Capital Normal University, Beijing, China}
\date{}

%

%

\global\long\def\CR{\mathrm{CR}}
\global\long\def\ST{\mathrm{ST}}
\global\long\def\SF{\mathrm{SF}}
\global\long\def\cov{\mathrm{cov}}
\global\long\def\dist{\mathrm{dist}}
\global\long\def\SLE{\mathrm{SLE}}
\global\long\def\hSLE{\mathrm{hSLE}}
\global\long\def\CLE{\mathrm{CLE}}
\global\long\def\GFF{\mathrm{GFF}}
\global\long\def\inte{\mathrm{int}}
\global\long\def\ext{\mathrm{ext}}
\global\long\def\inrad{\mathrm{inrad}}
\global\long\def\outrad{\mathrm{outrad}}
\global\long\def\dimH{\mathrm{dim}}
\global\long\def\capa{\mathrm{cap}}
\global\long\def\diam{\mathrm{diam}}
\global\long\def\free{\mathrm{free}}
\global\long\def\hF{{}_2\mathrm{F}_1}
\global\long\def\ghF{{}_3\mathrm{F}_2}
\global\long\def\simple{\mathrm{simple}}
\global\long\def\even{\mathrm{even}}
\global\long\def\odd{\mathrm{odd}}
\global\long\def\st{\mathrm{ST}}
\global\long\def\usf{\mathrm{USF}}
\global\long\def\Leb{\mathrm{Leb}}
\global\long\def\LP{\mathrm{LP}}
\global\long\def\coulomb{\LH}
\global\long\def\coulombnew{\LG}
\global\long\def\kfunc{p}
\global\long\def\OO{\mathcal{O}}
\global\long\def\Dist{\mathrm{Dist}}
\global\long\def\ball{D}

\global\long\def\eps{\epsilon}
\global\long\def\ov{\overline}
\global\long\def\U{\mathbb{U}}
\global\long\def\T{\mathbb{T}}
\global\long\def\HH{\mathbb{H}}
\global\long\def\LA{\mathcal{A}}
\global\long\def\LB{\mathcal{B}}
\global\long\def\LC{\mathcal{C}}
\global\long\def\LD{\mathcal{D}}
\global\long\def\LF{\mathcal{F}}
\global\long\def\LK{\mathcal{K}}
\global\long\def\LE{\mathcal{E}}
\global\long\def\LG{\mathcal{G}}
\global\long\def\LI{\mathcal{I}}
\global\long\def\LJ{\mathcal{J}}
\global\long\def\LL{\mathcal{L}}
\global\long\def\LM{\mathcal{M}}
\global\long\def\LN{\mathcal{N}}
\global\long\def\LQ{\mathcal{Q}}
\global\long\def\LR{\mathcal{R}}
\global\long\def\LT{\mathcal{T}}
\global\long\def\LS{\mathcal{S}}
\global\long\def\LU{\mathcal{U}}
\global\long\def\LV{\mathcal{V}}
\global\long\def\LW{\mathcal{W}}
\global\long\def\LX{\mathcal{X}}
\global\long\def\LY{\mathcal{Y}}
\global\long\def\PartF{\mathcal{Z}}
\global\long\def\LH{\mathcal{H}}
\global\long\def\LJ{\mathcal{J}}
\global\long\def\R{\mathbb{R}}
\global\long\def\C{\mathbb{C}}
\global\long\def\N{\mathbb{N}}
\global\long\def\Z{\mathbb{Z}}
\global\long\def\E{\mathbb{E}}
\global\long\def\PP{\mathbb{P}}
\global\long\def\QQ{\mathbb{Q}}
\global\long\def\A{\mathbb{A}}
\global\long\def\one{\mathbb{1}}
\global\long\def\bn{\mathbf{n}}
\global\long\def\MR{MR}
\global\long\def\cond{\,|\,}
\global\long\def\la{\langle}
\global\long\def\ra{\rangle}
\global\long\def\tree{\Upsilon}
\global\long\def\prob{\mathbb{P}}
\global\long\def\hm{\mathrm{Hm}}
\global\long\def\cross{\mathrm{Cross}}

\global\long\def\sf{\mathrm{SF}}
\global\long\def\wr{\varrho}

\global\long\def\Im{\operatorname{Im}}
\global\long\def\Re{\operatorname{Re}}

\global\long\def\ud{\mathrm{d}}
\global\long\def\pder#1{\frac{\partial}{\partial#1}}
\global\long\def\pdder#1{\frac{\partial^{2}}{\partial#1^{2}}}
\global\long\def\der#1{\frac{\ud}{\ud#1}}

\global\long\def\bZnn{\mathbb{Z}_{\geq 0}}

\global\long\def\Vfunc{\LG}
\global\long\def\gfunc{g^{(\rr)}}
\global\long\def\hfunc{h^{(\rr)}}

\global\long\def\SimplexInt{\rho}
\global\long\def\CubeInt{\widetilde{\rho}}

\global\long\def\ii{\mathfrak{i}}
\global\long\def\rr{\mathfrak{r}}
\global\long\def\chamber{\mathfrak{X}}
\global\long\def\Wchamber{\mathfrak{W}}

\global\long\def\SimplexIntKappa8{\SimplexInt}

\global\long\def\nested{\boldsymbol{\underline{\Cap}}}
\global\long\def\unnested{\boldsymbol{\underline{\cap\cap}}}
\global\long\def\unnested{\boldsymbol{\underline{\cap\cap}}}

\global\long\def\acycle{\vartheta}
\global\long\def\bcycle{\tilde{\acycle}}
\global\long\def\Gloop{\Theta}

\global\long\def\metric{\mathrm{dist}}

\global\long\def\adj#1{\mathrm{adj}(#1)}

\global\long\def\bs{\boldsymbol}

\global\long\def\edge#1#2{\langle #1,#2 \rangle}
\global\long\def\graph{G}

\newcommand{\conn}{\vartheta}
\newcommand{\hatconn}{\widehat{\vartheta}_{\mathrm{RCM}}}
\newcommand{\realpt}{\smash{\mathring{x}}}
\newcommand{\corrind}{\LC}
\newcommand{\bssymb}{\pi}
\newcommand{\PRCM}{\mu}
\newcommand{\coeff}{p}
\newcommand{\MainConst}{C}

\global\long\def\removeLink{/}
\maketitle

\begin{center}
\begin{minipage}{0.95\textwidth}
\abstract{
We consider global 2-SLE$_{\kappa}$ $(\eta_1, \eta_2)$ in a topological rectangle with $\kappa\in (4,8)$. 
We derive the law of a random hitting point of the curves and show that, conditional on this random hitting point, the two curves in the pair become $\SLE_{\kappa}(2,2,-4)$ and $\SLE_{\kappa}(2, -4)$ processes. Using a similar idea, we derive the asymptotics of the probability for $\eta_1\cap\eta_2=\emptyset$. As an application, we derive the asymptotics of the probability for the existence of two disjoint open paths in critical FK-Ising model. }

\bigskip{}


\end{minipage}
\end{center}

\tableofcontents
\newpage

\section{Introduction}
In 1999, Schramm introduced Schramm-Loewner evolution (SLE) as a candidate for the scaling limit of interfaces in planar critical statistical models. Since then, there are several models whose interfaces are proved to converge to the SLE processes: loop-erased random walk and uniform spanning tree~\cite{LawlerSchrammWernerLERWUST}, percolation~\cite{SmirnovPercolationConformalInvariance}, level lines of Gaussian free field (GFF)~\cite{SchrammSheffieldDiscreteGFF}, Ising and FK-Ising model~\cite{ChelkakSmirnovIsing, CDCHKSConvergenceIsingSLE}. In these cases, one considers the models in a simply connected domain with two marked points on the boundary. 
Such a domain is called a Dobrushin domain.
One proves that the interface connecting the two marked points in the Dobrushin domain converges to SLE processes. 
Now, we consider the models in a topological rectangle where there are four marked boundary points and the boundary condition of the model is alternating. In such a case, there is a pair of interfaces connecting among the four marked points. It turns out that the pair of two interfaces will converge to a pair of SLE curves, which we call global 2-$\SLE$ (see the precise definition in Definition~\ref{def::2SLE}).
Global 2-SLE has been studied in previous works  (see, e.g.,~\cite{MillerSheffieldWernerNonsimpleSLE, WuHyperSLE, ZhanGreen2SLE}). 
In this article, we focus on global 2-SLE with $\kappa\in (4,8)$. 
The marginal law of one curve in global 2-SLE$_{\kappa}$ is a variant of SLE process whose driving function involves hypergeometric function. 
In Theorem~\ref{thm::2SLE_decomposition}, we show that global 2-$\SLE_{\kappa}$ with $\kappa\in (4,8)$ can be decomposed as a random point and a pair of $\SLE_\kappa(\rho)$ processes with multiple force points. Using a similar idea for the proof of the decomposition, we are able to derive the asymptotics of the probability that the two curves in global 2-$\SLE_{\kappa}$ avoid each other, see Proposition~\ref{prop::global2_mono}. Such asymptotics can be applied to critical FK-Ising model and we derive the asymptotics of probability for the existence of two disjoint open paths in FK-Ising model, see Proposition~\ref{prop::FKIsing_mono}. 

\subsection{Global 2-SLE}
\label{subsec::intro_NSLE}

We first introduce polygons and the space of curves, and then give a formal definition of global 2-SLE. 
Fix $n\ge 1$, we call $(\Omega; x_1, \ldots, x_{n})$ a (topological) polygon if $\Omega\subset\C$ is a simply connected domain such that $\partial\Omega$ is locally connected and $x_1, \ldots, x_{n}$  are $n$ distinct points  lying along the boundary $\partial\Omega$ in counterclockwise order.
When $n=2$, we call such polygon a Dobrushin domain. For a Dobrushin domain $(\Omega; x_1, x_2)$, we denote by $(x_1x_2)$ the boundary arc from $x_1$ to $x_2$  in counterclockwise order which does not contain $\{x_1,x_2\}$. We will denote by $[x_1x_2]:=\{x_1,x_2\}\cup (x_1 x_2)$. 
When $n=4$, we call such a polygon a quad. 

For two continuous mappings $\eta,\tilde{\eta}$ from $[0,1]$ to $\mathbb{C}$, we define 
\begin{equation}\label{eqn::curve_metric}
	\dist(\eta, \tilde{\eta}):=\inf_{\varphi,\tilde{\varphi}}\sup_{t\in [0,1]}|\eta(\varphi(t))-\tilde{\eta}(\tilde{\varphi}(t))|,
\end{equation}
where the infimum is taken over all increasing homeomorphisms $\varphi, \tilde{\varphi}: [0,1]\to [0,1]$.
Define an equivalent relation on continuous mappings from $[0,1]$ to $\mathbb{C}$: $\eta\sim \tilde{\eta}$ if $\dist(\eta,\tilde{\eta})=0$. We denote by $X$ the set of planar oriented curves, i.e. continuous mappings from $[0,1]$ to $\C$ modulo the equivalent relation $\sim$. We equip $X$ with the metric~\eqref{eqn::curve_metric}.
 Then the metric space $(X, \dist)$ is complete and separable. 
We denote by $X_{\simple}(\Omega; x_1, x_2)$ the set of continuous simple unparameterized curves in $\Omega$ connecting $x_1$ and $x_2$ such that they only touch the boundary $\partial\Omega$ in $\{x_1, x_2\}$. We denote by $X_0(\Omega; x_1, x_2)$ the closure of $X_{\simple}(\Omega; x_1, x_2)$ in the metric topology $(X, \dist)$. Note that the curves in $X_0(\Omega; x_1, x_2)$ may have self-intersections. In particular, the chordal $\SLE_{\kappa}$ curves belongs to this space almost surely when $\kappa>4$. 

Fix a quad $(\Omega; x_1, x_2, x_3, x_4)$, we consider a pair of two curves $(\eta_1, \eta_2)$ in $\Omega$ such that each curve connects two points among $\{x_1, x_2, x_3, x_4\}$. These two curves can have two planar connectivities, but they are the same up to rotation. For definiteness, we focus on the pair of curves $(\eta_1, \eta_2)$ such that $\eta_1\in X_0(\Omega; x_1, x_2)$ and $\eta_2\in X_0(\Omega; x_3, x_4)$ such that $\eta_1\cap [x_3x_4]=\emptyset$ and $\eta_2\cap[x_1x_2]=\emptyset$ and that $\eta_1$ and $\eta_2$ do not cross. 
We denote the space of such pairs of two curves by 
$X_0(\Omega; x_1, x_2, x_3, x_4)$.

\begin{definition}\label{def::2SLE}
Fix $\kappa\in (0,8)$ and a quad $(\Omega; x_1, x_2, x_3, x_4)$. 
We call a probability measure on pairs $(\eta_1, \eta_2)\in X_0(\Omega; x_1, x_2, x_3, x_4)$ a global $2$-$\SLE_{\kappa}$ if the conditional law of $\eta_1$ given $\eta_2$ is $\SLE_{\kappa}$ connecting $x_1$ and $x_2$ in the connected component of $\Omega\setminus\eta_2$ having $x_1, x_2$ on its boundary; 
and the conditional law of $\eta_2$ given $\eta_1$ is $\SLE_{\kappa}$ connecting $x_3$ and $x_4$ in the connected component of $\Omega\setminus\eta_1$ having $x_3, x_4$ on its boundary.
\end{definition}
The existence and uniqueness of global $2$-$\SLE_{\kappa}$ with $\kappa\in (0,8)$ are known, see~\cite[Theorem~A.1]{MillerSheffieldWernerNonsimpleSLE} and~\cite[Proposition~6.10]{WuHyperSLE}. The notion of global 2-SLE can be generalized to global multiple SLEs, see related conclusions in~\cite{KozdronLawlerMultipleSLEs, PeltolaWuGlobalMultipleSLEs, BeffaraPeltolaWuUniqueness, ZhanExistenceUniquenessMultipleSLE, AngHoldenSunYu2023, FengLiuPeltolaWu2024,ambrosio2025multiplesle}.  {In~\cite{FengLiuPeltolaWu2024}, the authors generalise the constructions in this paper and give the construction of pure partition functions for $\SLE_\kappa$ and global $N$-$\SLE_{\kappa}$ when $\kappa\in (4,8)$ and $N\ge 2$.}

\subsection{Decomposition of global 2-SLE for $\kappa\in (4,8)$}
Our first conclusion is a decomposition of global $2$-$\SLE_{\kappa}$: it can be decomposed into a random hitting point and a pair of $\SLE_\kappa(\rho)$ processes with multiple force points {(see Section~\ref{sec::pre})}.

\begin{theorem}\label{thm::2SLE_decomposition}
Fix $\kappa\in (4,8)$ and a quad $(\Omega; x_1, x_2, x_3, x_4)$ and suppose $(\eta_1, \eta_2)\in X_0(\Omega; x_1, x_2, x_3, x_4)$ is a global $2$-$\SLE_{\kappa}$ such that $\eta_1$ goes from $x_2$ to $x_1$ and $\eta_2$ goes from $x_3$ to $x_4$. Let $\tau_1$ (resp. $\tau_2$) be the first time that $\eta_1$ (resp. $\eta_2$) hits the boundary arc $(x_4x_1)$ and denote $U_1=\eta_1(\tau_1)$ and $U_2=\eta_2(\tau_2)$. Then we have the following characterizations. See Figure~\ref{fig::intro_decomp}. 
\begin{itemize}
\item 
We fix a conformal map $\psi$ from $\Omega$ onto $\HH$ such that $\psi(x_1)=\infty$. Then the law of $\psi(U_2)=\psi(\eta_2(\tau_2))$ is absolutely continuous with respect to the Lebesgue measure on the interval $(\psi(x_4), \infty)$ and its density is given by 
\begin{equation}\label{eqn::2SLE_hittingpoint_density}
r(u)=\frac{\prod_{j=2}^4(u-\psi(x_j))^{-4/\kappa}}{\int_{\psi(x_4)}^{\infty}\prod_{j=2}^4(v-\psi(x_j))^{-4/\kappa}\ud v}. 
\end{equation}
\item Given $U_2$, the conditional law of $\eta_1$ is $\SLE_{\kappa}(2,2,-4)$ in $\Omega$ from $x_2$ to $x_1$ with force points $(x_3, x_4, U_2)$. 
\item Given $(U_2, \eta_1)$, denote by $\Omega^R$ the connected component of $\Omega\setminus\eta_1$ having $(x_3x_4)$ on the boundary. Then the conditional law of $\eta_2[0,\tau_2]$ given $(U_2, \eta_1)$ is $\SLE_{\kappa}(2, -4)$ in $\Omega^R$ from $x_3$ to $U_1$ with force points $(x_4, U_2)$, up to the first hitting time of $(x_4x_1)$. Given $\eta_2[0,\tau_2]$, the curve $(\eta_2(t): t\ge \tau_2)$ is an independent $\SLE_{\kappa}$ from $U_2$ to $x_4$ in the connected component of $\Omega\setminus\eta_2[0,\tau_2]$ having $x_4$ on the boundary.  
\end{itemize}
\end{theorem}

\begin{figure}[ht!]
\begin{center}
\includegraphics[width=0.4\textwidth]{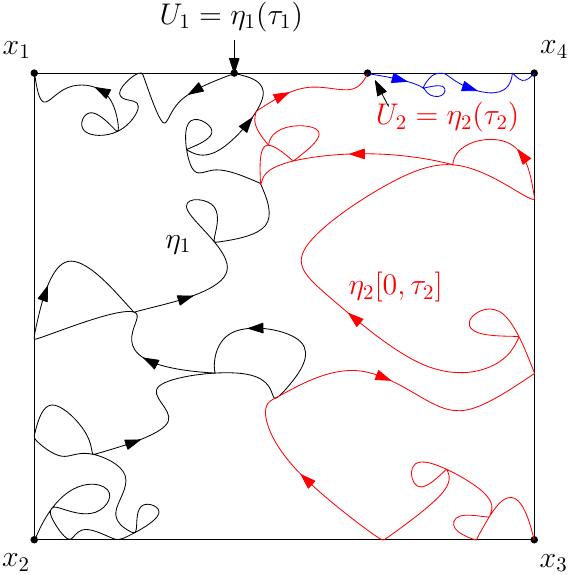}
\end{center}
\caption{\label{fig::intro_decomp}
Given $U_2$, the conditional law of $\eta_1$ is $\SLE_{\kappa}(2,2,-4)$ in $\Omega$ from $x_2$ to $x_1$ with force points $(x_3, x_4, U_2)$. 
Given $(U_2, \eta_1)$, the conditional law of $\eta_2[0,\tau_2]$ is $\SLE_{\kappa}(2,-4)$ in $\Omega^R$ from $x_3$ to $U_1$ with force points $(x_4, U_2)$. }
\end{figure}

The density in~\eqref{eqn::2SLE_hittingpoint_density} is reminiscent of Schwarz-Christoffel formula. Indeed, Schwarz-Christoffel formula gives an equivalent description of the law of the random point $U$ which we present in the following corollary. 

\begin{corollary}\label{cor::2SLE_quad}
Assume the same setup as in Theorem~\ref{thm::2SLE_decomposition}. 
Let $\phi$ be a conformal map from $\Omega$ onto a quadrilateral with four corners $(\phi(x_1), \phi(x_2), \phi(x_3), \phi(x_4))$ and the interior angles are 
\begin{equation}\label{eqn::2SLE_quad_angles}
\pi(12/\kappa-1)\text{ at }\phi(x_1),\quad \pi(1-4/\kappa)\text{ at }\phi(x_j)\text{ for }j=2,3,4.
\end{equation}
See Figure~\ref{fig::2SLE_quad}. Then the law of $\phi(\eta_2(\tau_2))$ is uniform over $(\phi(x_4)\phi(x_1))$. 
\end{corollary}

\begin{figure}[ht!]
\begin{subfigure}[b]{\textwidth}
\begin{center}
\includegraphics[width=0.3\textwidth]{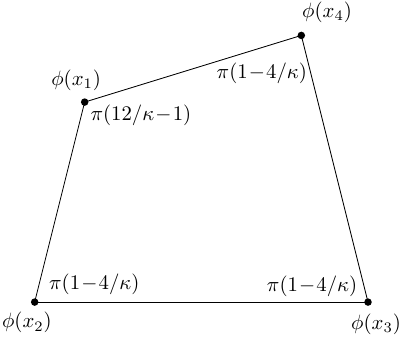}
\end{center}
\caption{A quadrilateral with four corners $(\phi(x_1), \phi(x_2), \phi(x_3), \phi(x_4))$ whose interior angles are given by~\eqref{eqn::2SLE_quad_angles}. }
\end{subfigure}\\
\vspace{1cm}
\begin{subfigure}[b]{0.3\textwidth}
\begin{center}
\includegraphics[width=\textwidth]{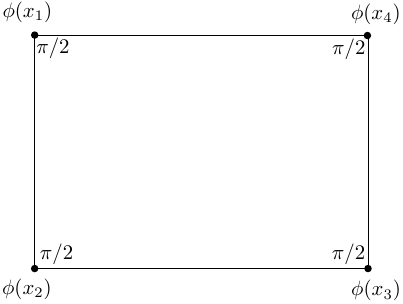}
\end{center}
\caption{$\kappa=8$.}
\end{subfigure}
$\quad$
\begin{subfigure}[b]{0.3\textwidth}
\begin{center}
\includegraphics[width=\textwidth]{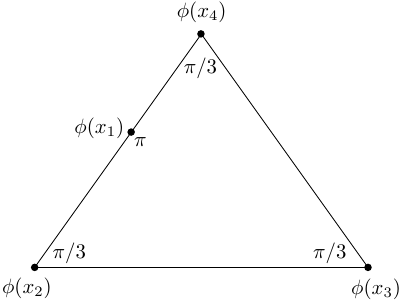}
\end{center}
\caption{$\kappa=6$.}
\end{subfigure}
$\quad$
\begin{subfigure}[b]{0.3\textwidth}
\begin{center}
\includegraphics[width=\textwidth]{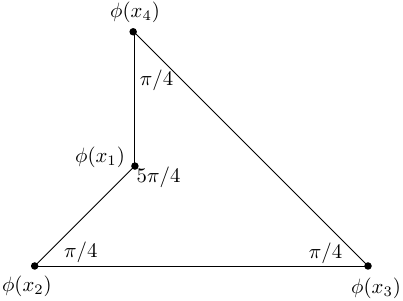}
\end{center}
\caption{$\kappa=16/3$.}
\end{subfigure}
\caption{\label{fig::2SLE_quad} 
Let $\phi$ be a conformal map from $\Omega$ onto a quadrilateral as in (a). The law of $\phi(\eta_2(\tau_2))$ is uniform over $(\phi(x_4)\phi(x_1))$. }
\end{figure}

Corollary~\ref{cor::2SLE_quad} is consistent with known results when $\kappa=6$ (Cardy's formula, proved in~\cite{SmirnovPercolationConformalInvariance}) and $\kappa=8$ (Proposition 1.3 of~\cite{HanLiuWuUST}), but we think the conclusion for general $\kappa\in (4,8)$ never appears in the literature before. See more discussion for Corollary~\ref{cor::2SLE_quad} in Section~\ref{subsec::2SLE_discussion}. 

\smallbreak

Next, we will give an application of the above decomposition of global 2-SLE. Suppose $(\eta_1, \eta_2)$ is global 2-$\SLE_{\kappa}$ with $\kappa\in (4,8)$. We are interested in the probability for  $\eta_1\cap\eta_2=\emptyset$. This probability is a conformally invariant quantity of the quad $(\Omega; x_1, x_2, x_3, x_4)$. Although we are not able to derive the explicit formula for such a probability, we could give the asymptotic of it when the quad is ``very tall". To this end, we first introduce extremal distance of quads. For a quad $(\Omega; x_1, x_2, x_3, x_4)$, there exists a unique $L>0$ and a unique conformal map $\phi$ from $\Omega$ onto the rectangle $[0,1]\times[0,\ii L/\pi]$ which sends $(x_1, x_2, x_3, x_4)$ to the four corners $(\ii L/\pi, 0, 1, 1+\ii L/\pi)$. We call $L=L(\Omega; x_1, x_2, x_3, x_4)$ the extremal distance between $(x_2x_3)$ and $(x_4x_1)$ in $\Omega$.\footnote{The definition of extremal distance here differs from the standard definition by a factor $\pi$. Our choice here will simplify notations in this article.} We derive the asymptotic of the probability for $\eta_1\cap\eta_2=\emptyset$ when $L\to\infty$ in the following proposition. 

\begin{proposition}\label{prop::global2_mono}
Fix $\kappa\in (4,8)$ and a quad $(\Omega; x_1, x_2, x_3, x_4)$ and suppose $(\eta_1, \eta_2)\in X_0(\Omega; x_1, x_2, x_3, x_4)$ is a global $2$-$\SLE_{\kappa}$ such that $\eta_1$ goes from $x_2$ to $x_1$ and $\eta_2$ goes from $x_3$ to $x_4$.
Let $L$ be the extremal distance between $(x_2x_3)$ and $(x_4x_1)$ in $\Omega$. 
From the conformal invariance, the probability for $\eta_1\cap \eta_2=\emptyset$ only depends on $L$ and we denote it by $p^{(2)}_{\kappa}(L)$. 
Then we have\footnote{For two functions $p(L)$ and $q(L)$, we write $p(L)\asymp q(L)$ as $L\to\infty$ if there exist constants $C\in (1,\infty)$ and $L_0\in (0,\infty)$ such that $C^{-1}\le p(L)/q(L)\le C$ for all $L\ge L_0$. }
\begin{equation}\label{eqn::continuousestimate}
p^{(2)}_{\kappa}(L)\asymp\exp(-(2-8/\kappa)L),\quad \text{as }L\to\infty.
\end{equation}
\end{proposition}

The asymptotic in Proposition~\ref{prop::global2_mono} will give us the asymptotic of the probability of the existence of two disjoint open paths in random-cluster model that we specify below. 

\subsection{Application to random-cluster models}
We fix parameters 
\begin{equation}\label{eqn::kappa_q}
\kappa\in (4,8),\qquad 
q= 4\cos^2\left(\frac{4\pi}{\kappa}\right)\in (0,4).
\end{equation}
We fix a quad $(\Omega; x_1, x_2, x_3, x_4)$ and suppose $(\Omega^{\delta}; x_1^{\delta}, x_2^{\delta}, x_3^{\delta}, x_4^{\delta})$ is a sequence of discrete quads that converges to $(\Omega; x_1, x_2, x_3, x_4)$ in the close-Carath\'{e}odory sense (see Section~\ref{sec::FKIsing_mono}). We consider critical random-cluster model with parameters $q\in (0,4)$ and $p=p_c(q)=\frac{\sqrt{q}}{1+\sqrt{q}}$ on $(\Omega^{\delta}; x_1^{\delta}, x_2^{\delta}, x_3^{\delta}, x_4^{\delta})$ with alternating boundary conditions: 
\begin{equation} \label{eqn::bc}
(x_2^{\delta}x_3^{\delta})\text{ is wired,}\quad(x_4^{\delta}x_1^{\delta})\text{ is wired,}\quad (x_2^{\delta}x_3^{\delta})\text{ and }(x_4^{\delta}x_1^{\delta}) \text{ are not wired outside of }\Omega^{\delta}.
\end{equation}

We are interested in the probability that there exist one open crossing or two open crossings in $\Omega^{\delta}$. 
More concretely, we consider the following two crossing events:
\begin{itemize}
\item We denote by $\conn^{(1)}=\conn^{(1)}(\Omega^{\delta}; x_1^{\delta}, x_2^{\delta}, x_3^{\delta}, x_4^{\delta})$ the event that there exists an open path in $\Omega^{\delta}$ connecting the boundary arcs $(x_2^{\delta}x_3^{\delta})$ and $(x_4^{\delta}x_1^{\delta})$. 
\item We denote by $\conn^{(2)}=\conn^{(2)}(\Omega^{\delta}; x_1^{\delta}, x_2^{\delta}, x_3^{\delta}, x_4^{\delta})$ the event that there exist two disjoint open paths in $\Omega^{\delta}$ connecting the boundary arcs $(x_2^{\delta}x_3^{\delta})$ and $(x_4^{\delta}x_1^{\delta})$. 
\end{itemize}

We first give a summary on the scaling limit of the probability of $\conn^{(1)}$. 
Using conformal field theory (CFT) heuristic, the probability for $\conn^{(1)}$ is predicted in the physics literature, see~\cite{SimmonsKlebanFloresZiff2011, FloresSimmonsKlebanZiffCrossingProba} and references therein: 
\begin{align}\label{eqn::RCM_Cardy}
p^{(1)}_{\mathrm{RCM}}(\Omega; x_1, x_2, x_3, x_4):=&
\lim_{\delta\to 0}\PP_{\mathrm{RCM}}\left[\conn^{(1)}(\Omega^{\delta}; x_1^{\delta}, x_2^{\delta}, x_3^{\delta}, x_4^{\delta})\right]\notag\\
=&\frac{s^{8/\kappa-1}\hF(\frac{4}{\kappa}, 1-\frac{4}{\kappa}, \frac{8}{\kappa}; s)}{s^{8/\kappa-1}\hF(\frac{4}{\kappa}, 1-\frac{4}{\kappa}, \frac{8}{\kappa}; s)+\sqrt{q}(1-s)^{8/\kappa-1}\hF(\frac{4}{\kappa}, 1-\frac{4}{\kappa}, \frac{8}{\kappa}; 1-s)},
\end{align}
where $\hF$ is a hypergeometric function {(see~\cite[Section 15]{AbramowitzHandbook} for definitions and properties of hypergeometric functions)} and $s$ is the cross-ratio:
\[s=\frac{(\varphi(x_4)-\varphi(x_1))(\varphi(x_3)-\varphi(x_2))}{(\varphi(x_3)-\varphi(x_1))(\varphi(x_4)-\varphi(x_2))},\]
and $\varphi$ is any conformal map from $\Omega$ onto $\HH$ such that $\varphi(x_1)<\varphi(x_2)<\varphi(x_3)<\varphi(x_4)$. 
Formula~\eqref{eqn::RCM_Cardy} is proved to be true by Chelkak and Smirnov~\cite[Eq.~(1.1)]{ChelkakSmirnovIsing} for FK-Ising model with $q=2$ using miraculous complex analysis tools. 

We focus on the asymptotic of the probability in~\eqref{eqn::RCM_Cardy} when the quad $(\Omega; x_1, x_2, x_3, x_4)$ is very tall. Let $L$ be the extremal distance between $(x_2x_3)$ and $(x_4x_1)$ in $\Omega$. It follows from~\eqref{eqn::RCM_Cardy} and Lemma~\ref{lem::Ls} that\footnote{For two functions $p(L)$ and $q(L)$, we write $p(L)\sim q(L)$ as $L\to\infty$ if the limit  of $p(L)/q(L)$ exists as $L\to\infty$ and $\lim_{L\to\infty}p(L)/q(L)\in (0,\infty)$. }
\begin{equation}\label{eqn::RCM_Cardy_asy}
p^{(1)}_{\mathrm{RCM}}(L):=p^{(1)}_{\mathrm{RCM}}(\Omega; x_1, x_2, x_3, x_4)\sim\exp(-(8/\kappa-1)L), \quad\text{as }L\to\infty.
\end{equation}

Next, let us explain our conclusion for $\conn^{(2)}$. In a similar way as above, we denote 
\begin{align}\label{eqn::RCM_mono}
p^{(2)}_{\mathrm{RCM}}(L):=p^{(2)}_{\mathrm{RCM}}(\Omega; x_1, x_2, x_3, x_4):=&
\lim_{\delta\to 0}\PP_{\mathrm{RCM}}^{\delta}\left[\conn^{(2)}(\Omega^{\delta}; x_1^{\delta}, x_2^{\delta}, x_3^{\delta}, x_4^{\delta})\right].
\end{align}
Then, the asymptotic in~\eqref{eqn::RCM_Cardy_asy} and Proposition~\ref{prop::global2_mono} give the following prediction: 
\begin{equation}\label{eqn::RCM_mono_asy}
p^{(2)}_{\mathrm{RCM}}(L)\asymp\exp(-L),\quad\text{as }L\to\infty.
\end{equation}
In particular, the asymptotic in~\eqref{eqn::RCM_mono_asy} is true for FK-Ising model and we wrap up our conclusion in Proposition~\ref{prop::FKIsing_mono}.

\begin{proposition}\label{prop::FKIsing_mono}
Fix a quad $(\Omega; x_1, x_2, x_3, x_4)$. Let $L$ be the extremal distance between $(x_2x_3)$ and $(x_4x_1)$ in $\Omega$. Suppose $(\Omega^{\delta}; x_1^{\delta}, x_2^{\delta}, x_3^{\delta}, x_4^{\delta})$ is a sequence of discrete quads that converges to $(\Omega; x_1, x_2, x_3, x_4)$ in the close-Carath\'{e}odory sense. The limit in~\eqref{eqn::RCM_mono} exists and is conformally invariant for FK-Ising model ($q=2$ and $\kappa=16/3$). Furthermore, we have 
\begin{equation}\label{eqn::FKIsing_mono}
p^{(2)}_{\mathrm{FKIsing}}(L)\asymp\exp(-L),\quad\text{as }L\to\infty.
\end{equation}
\end{proposition}

\begin{remark}
We note that the asymptotic in~\eqref{eqn::RCM_mono_asy} also holds for critical Bernoulli site percolation on triangular lattice, using a similar proof for Proposition~\ref{prop::FKIsing_mono}. 
\end{remark}


\paragraph*{Outline and strategy.}
The rest of this article is organized as follows. 
We give preliminaries in Section~\ref{sec::pre}. 
We prove Theorem~\ref{thm::2SLE_decomposition}
 and Corollary~\ref{cor::2SLE_quad} in Section~\ref{sec::global2_deomposition}. 
We prove Proposition~\ref{prop::global2_mono} in Section~\ref{sec::global2_mono}. 
The proofs for Theorem~\ref{thm::2SLE_decomposition} and Proposition~\ref{prop::global2_mono} in Sections~\ref{sec::global2_deomposition}-\ref{sec::global2_mono} rely on special martingale observables (Lemma~\ref{lem::mart1} and Lemma~\ref{lem::aux1}). To construct such observables, we are inspired by Coulomb gas formalism of SLE partition functions in CFT, but our proof does not use any background of CFT. We prove Proposition~\ref{prop::FKIsing_mono} in Section~\ref{sec::FKIsing_mono}. Roughly speaking, we will show in Section~\ref{sec::FKIsing_mono} that the probability $p_{\mathrm{RCM}}^{(2)}(L)$ in~\eqref{eqn::RCM_mono} is the product of $p_{\mathrm{RCM}}^{(1)}(L)$ in~\eqref{eqn::RCM_Cardy_asy} and $p_{\kappa}^{(2)}(L)$ in~\eqref{eqn::continuousestimate}, because conditional on $\conn^{(1)}$, the scaling limit of the pair of interfaces in random-cluster model converges to global 2-$\SLE_{\kappa}$ (Lemma~\ref{lem::cvg_global_sle}). 
In Appendix~\ref{appendix::crossratio_extremaldistance}, we derive the relation between the cross-ratio and the extremal distance of quads.

\paragraph*{Acknowledgments.}
The authors acknowledge two anonymous referees for helpful comments and suggestions. 
M.L. is supported by Knut and Alice Wallenberg Foundation.
Y. F. and H.W. are supported by Beijing Natural Science Foundation (JQ20001) and New Cornerstone Investigator Program 100001127. 
H.W. is partly affiliated at Yanqi Lake Beijing Institute of Mathematical Sciences and Applications, Beijing, China.

\section{Preliminaries}
\label{sec::pre}
\paragraph*{Gamma functions and Beta functions.}
In later sections, we will use the connection between Beta function and Gamma function. Fix $\kappa\in(4,8)$ and recall the definition of $q$ in~\eqref{eqn::kappa_q}. Recall that Beta function~\cite[Eq.~(6.2.1)]{AbramowitzHandbook} is given by 
\[B(z,w)=\int_0^1 t^{z-1}(1-t)^{w-1}\ud t.\]
To simplify notations, we denote 
\begin{equation}\label{eqn::constant_def}
C_\kappa=\frac{\Gamma(1-4/\kappa)^2}{\sqrt{q} \, \Gamma(2-8/\kappa)}. 
\end{equation}
Then we have
\begin{equation}\label{eqn::constC_Beta}
B(1-4/\kappa, 1-4/\kappa)=\sqrt{q}C_{\kappa},\quad 
B(1-4/\kappa,8/\kappa-1)=C_\kappa.
\end{equation}
This is because 
\begin{align*}
B(1-4/\kappa, 1-4/\kappa)=& \frac{\Gamma(1-4/\kappa)^2}{\Gamma(2-8/\kappa)};\tag{due to~\cite[Eq.~(6.2.2)]{AbramowitzHandbook}}\\
B(1-4/\kappa,8/\kappa-1)=&\frac{\Gamma(1-4/\kappa)\Gamma(8/\kappa-1)}{\Gamma(4/\kappa)}\tag{due to~\cite[Eq.~(6.2.2)]{AbramowitzHandbook}}\\
=&\frac{\Gamma(1-4/\kappa)^2}{\Gamma(2-8/\kappa)}\times\frac{\sin\left(4\pi/\kappa\right)}{\sin\left(\pi(8-\kappa)/\kappa\right)}\tag{due to~\cite[Eq.~(6.1.17)]{AbramowitzHandbook}}\\
=&\frac{\Gamma(1-4/\kappa)^2}{\Gamma(2-8/\kappa)\times(-2\cos(4\pi/\kappa))}. 
\end{align*}

\paragraph*{Poisson kernel.}
For a Dobrushin domain $(\Omega; x, y)$, we say it is nice if the boundary points $x, y$ lie on sufficiently regular segments of $\partial\Omega$ (e.g. $C^{1+\eps}$ for some $\eps>0$).
The (boundary) Poisson kernel $H(\Omega; x, y)$ is defined for {\color{blue}a }nice Dobrushin domain $(\Omega; x, y)$.  When $\Omega=\HH$, we have
\begin{equation}\label{eqn::Poisson_def_H}
H(\HH; x,y)=|y-x|^{-2},\quad x,y\in\R. 
\end{equation}
For a nice Dobrushin domain $(\Omega; x, y)$, we extend its definition via conformal covariance: 
\begin{equation}\label{eqn::Poisson_def_general}
H(\Omega; x, y)=|\varphi'(x)\varphi'(y)|H(\HH; \varphi(x), \varphi(y)), 
\end{equation}
where $\varphi$ is any conformal map from $\Omega\to \HH$.
Poisson kernel satisfies the following monotonicity.
\begin{lemma}\label{lem::Poisson_mono}
Let $(\Omega; x, y)$ be a nice Dobrushin domain and let $U\subset\Omega$ be a subdomain that agrees with $\Omega$ in a neighborhood of the boundary arc $(yx)$. Then we have
\[H(U; x, y)\le H(\Omega; x, y).\] 
\end{lemma} 
\begin{proof}
This follows from~\eqref{eqn::Poisson_def_H} and~\eqref{eqn::Poisson_def_general}. 
\end{proof}

\paragraph*{Loewner chains and SLE.}
We call a compact subset $K$ of $\overline{\HH}$ an $\HH$-hull if $\HH\setminus K$ is simply connected. Riemann's mapping theorem asserts that there exists a unique conformal map $g_K$ from $\HH\setminus K$ onto $\HH$ normalized at $\infty$: $\lim_{z\to\infty}|g_K(z)-z|=0$. We consider a collection of growing $\HH$-hulls. They are associated with families of conformal maps $(g_t, t\ge 0)$ by solving the Loewner equation: for $z\in\overline{\HH}$, 
\[\partial_t g_t(z)=\frac{2}{g_t(z)-W_t},\quad g_0(z)=z, \]
where $(W_t, t\ge 0)$ is a real-valued continuous function, which we call driving function. We define $\tau(z)$ to be the swallowing time of $z$, i.e. $\sup\{t\ge 0: \min_{s\in [0,t]}|g_s(z)-W_s|>0\}$ and write $K_t=\{z\in\HH: \tau(z)\le t\}$. Then $g_t$ is the unique conformal map from $\HH\setminus K_t$ onto $\HH$ normalized at $\infty$. The collection of $\HH$-hulls $(K_t, t\ge 0)$ is called a Loewner chain. 

For $\kappa\ge0$, (chordal) $\SLE_{\kappa}$ in $\HH$ from $0$
to $\infty$ is a random Loewner chain driven by $W_t=\sqrt{\kappa}B_t$ where $(B_t, t\ge 0)$ is a standard one-dimensional Brownian motion. 
Rohde and Schramm~\cite{RohdeSchrammSLEBasicProperty} prove that there exists a curve $\gamma$ such that $\HH\setminus K_t$ is the same as the unbounded connected component of $\HH\setminus\gamma[0,t]$. We call this curve $\SLE_{\kappa}$ curve in $\HH$ from $0$ to $\infty$. Such  a curve is simple when $\kappa\in (0,4]$ and self-touching when $\kappa\in (4,8)$. 
 For a Dobrushin domain $(\Omega; x, y)$, an $\SLE_{\kappa}$ in $\Omega$ from $x$ to $y$ can be defined via conformal image: suppose $\varphi$ is any conformal map from $\Omega$ onto $\HH$ such that $\varphi(x)=0$ and $\varphi(y)=\infty$, and then we define $\SLE_{\kappa}$ in $\Omega$ from $x$ to $y$ to be the $\varphi^{-1}$ image of an $\SLE_{\kappa}$ in $\HH$ from $0$ to $\infty$. 

We record a technical lemma which we will frequently use in later sections. 
\begin{lemma}\label{lem::technical}
Fix $y_0<y_1<y_2<v<u$. Suppose $\gamma$ is  a continuous curve starting from $y_0$ which admits a continuous driving function. 
Denote by $W$ the driving function and by $(g_t, t\ge 0)$ the corresponding conformal maps. 
Suppose $\gamma\cap[y_1, v]=\emptyset$ and $\gamma$ hits $(v,+\infty)$ at $u$. Denote by $T$ the hitting time of $\gamma$ at $(y_2,+\infty)$. Then, we have 
\begin{equation}\label{eqn::ratio_estimate1}
\lim_{t\to T}\frac{g_t(y_2)-W_t}{g_t(y_1)-W_t}=1\quad\text{and}\quad\lim_{t\to T}\frac{g_t(u)-g_t(y_2)}{g_t(u)-g_t(y_1)}=1.
\end{equation}
Moreover, denote by $\LC_\gamma$ the connected component of $\HH\setminus\gamma$ whose boundary contains $v$ and choose any conformal map $\xi_\gamma$ from $\LC_\gamma$ onto $\HH$ such that $\xi_\gamma(\gamma(T))=\infty$.Then, we have
\begin{equation}\label{eqn::ratio_estimate2}
\lim_{t\to T}\frac{g_t'(v)^2}{g_t'(y_1)g_t'(y_2)}=\frac{\xi'_\gamma(v)^2}{\xi'_\gamma(y_1)\xi'_\gamma(y_2)}.
\end{equation}
\end{lemma}
\begin{proof} 
For~\eqref{eqn::ratio_estimate1}, this is a standard estimate, see e.g.~\cite[Equation (4.6)]{LiuWuCrossingProbaMGFF}.
{For~\eqref{eqn::ratio_estimate2}, we define $\xi_t$ to be the conformal map from the unbounded component of $\HH\setminus\gamma[0,t]$ onto $\HH$ such that $\xi_t(y_1)=0$, $\xi_t(y_2)=1$ and $\xi_t(\gamma(T))=\infty$:
\[\xi_t(z):=\frac{g_t(z)-g_t(y_1)}{g_t(\gamma(T))-g_t(z)}\times \frac{g_t(\gamma(T))-g_t(y_2)}{g_t(y_2)-g_t(y_1)}.\] 
Then, we have
\[\frac{g_t'(v)^2}{g_t'(y_1)g_t'(y_2)}=\frac{\xi'_t(v)^2}{\xi'_t(y_1)\xi'_t(y_2)}\times\frac{(g_t(\gamma(T))-g_t(v))^2}{(g_t(\gamma(T))-g_t(y_1))(g_t(\gamma(T))-g_t(y_2))}.\]
Thus, from~\eqref{eqn::ratio_estimate1}, we have
\[
\lim_{t\to T}\frac{g_t'(v)^2}{g_t'(y_1)g_t'(y_2)}=
\lim_{t\to T}\frac{\xi'_t(v)^2}{\xi'_t(y_1)\xi'_t(y_2)}=\frac{\xi'_{\gamma}(v)^2}{\xi'_{\gamma}(y_1)\xi'_{\gamma}(y_2)}.
\]
Note that the right side of~\eqref{eqn::ratio_estimate2} is independent of the choice of $\xi_{\gamma}$ as long as $\xi_{\gamma}(\gamma(T))=\infty$. This completes the proof.
}
\end{proof}

\paragraph*{Loewner chains associated with partition functions.}
 We denote
\[\chamber_{n}:=\{(x_1, \ldots, x_{n})\in\R^{n}: x_1<\cdots<x_{n}\}.\] 
A partition function is a positive smooth function 
$\PartF: \chamber_n\to \R$
satisfying the following two properties: for some fixed $i\in \{1, \ldots, n\}$ and scaling exponents $(\Delta_j, 1\le j\le n)$ with $\Delta_j\in\R$, 
\begin{itemize}
\item Partial differential equation: 
\begin{equation}\label{eqn::PDE_general}
\left[\frac{\kappa}{2}\partial_i^2+\sum_{j\neq i}\left(\frac{2}{x_j-x_i}\partial_j-\frac{2\Delta_j}{(x_j-x_i)^2}\right)\right]\PartF(x_1, \ldots, x_n)=0.
\end{equation}
\item M\"{o}bius covariance: for all M\"{o}bius maps $\varphi$ of $\HH$ such that $\varphi(x_1)<\cdots<\varphi(x_n)$, 
\begin{equation}\label{eqn::COV_general}
\PartF(x_1, \ldots, x_n)=\prod_{j=1}^n \varphi'(x_j)^{\Delta_j}\times \PartF(\varphi(x_1), \ldots, \varphi(x_n)). 
\end{equation} 
\end{itemize} 

Suppose $\gamma$ is an $\SLE_{\kappa}$ in $\HH$ from $x_i$ to $\infty$, and denote by $(W_t, t\ge 0)$ its driving function and by $(g_t, t\ge 0)$ the corresponding family of conformal maps. As $\PartF$ satisfies~\eqref{eqn::PDE_general}, the following process is a local martingale for $\gamma$: 
\begin{equation}\label{eqn::mart_general}
M_t(\PartF):=\prod_{j\neq i} g_t'(x_j)^{\Delta_j}\times \PartF(g_t(x_1), \ldots, g_t(x_{i-1}), W_t, g_t(x_{i+1}), \ldots, g_t(x_n)). 
\end{equation}
We call Loewner chain associated to partition function $\PartF$ 
starting from $x_i$ (with launching points $(x_1, \ldots, x_n)$)
the process $\gamma$ weighted by the local martingale~\eqref{eqn::mart_general}. 
This process is well-defined up to the swallowing time of $x_{i-1}$ or $x_{i+1}$. 
For the Loewner chain associated with $\PartF$ starting from $x_i$ 
(with launching points $(x_1, \ldots, x_n)$), its driving function $W_t$ is the solution to the following system of SDEs: 
\begin{equation}
\begin{cases}
\ud W_t=\sqrt{\kappa}\ud B_t+\kappa(\partial_i\log\PartF)(g_t(x_1), \ldots, g_t(x_{i-1}), W_t, g_t(x_{i+1}), \ldots, g_t(x_n))\ud t, \quad W_0=x_i; \\
\partial_t g_t(x_j)=\frac{2}{g_t(x_j)-W_t},\quad g_0(x_j)=x_j,\quad \text{for }j\neq i. 
\end{cases}
\end{equation}
For a general polygon $(\Omega; x_1, \ldots, x_n, y)$, we define Loewner chain associated with $\PartF$ in $\Omega$ from $x_i$ to $y$ via conformal image: suppose $\varphi$ is any conformal map from $\Omega$ onto $\HH$ such that $\varphi(x_1)<\cdots<\varphi(x_n)$ and $\varphi(y)=\infty$, then it is the $\varphi^{-1}$ image of the Loewner chain associated with $\PartF$
starting from $\varphi(x_i)$ (with launching points $(\varphi(x_1), \ldots, \varphi(x_n))$).  

Note that $\frac{6-\kappa}{2\kappa}$ is a very special scaling exponent and in the following, we denote 
\[h=\frac{6-\kappa}{2\kappa}.\]
{An $\SLE_{\kappa}$ in $\HH$ from $x$ to $y$ is the same as $\SLE_\kappa(\kappa-6)$ in $\HH$ from $x$ to $\infty$ with force point $y$ (see~\cite[Theorem 3]{SchrammWilsonSLECoordinatechanges})}, thus its Loewner chain is associated to the partition function
\begin{equation*}
\PartF(x, y)=|x-y|^{-2h}.
\end{equation*}

An $\SLE_{\kappa}(\rho)$ process is a generalization of $\SLE_{\kappa}$ in which one keeps track of additional marked points. Fix 
$x_1<\cdots<x_n$ and $\rho_1, \ldots, \rho_n\in\R$, the following function 
\begin{equation}\label{eqn::partf_slekapparho}
\PartF(x_1, \ldots, x_n):=\prod_{j\neq i}|x_j-x_i|^{\rho_j/\kappa}\times \prod_{j\neq i, r\neq i, j<r}|x_j-x_r|^{\rho_j\rho_r/(2\kappa)}
\end{equation}
satisfies~\eqref{eqn::PDE_general} and~\eqref{eqn::COV_general} with 
\begin{equation*}
\Delta_i=h,\quad \text{and}\quad \Delta_j=\frac{\rho_j(\rho_j+4-\kappa)}{4\kappa},\quad\text{for }j\neq i.
\end{equation*}
The Loewner chain associated with $\PartF$ in~\eqref{eqn::partf_slekapparho} is the so-called $\SLE_{\kappa}(\rho_1, \ldots, \rho_{i-1}; \rho_{i+1}, \ldots, \rho_n)$ process in $\HH$ from $x_i$ to $\infty$ with force points $(x_1, \ldots, x_{i-1}; x_{i+1}, \ldots, x_n)$.

\section{Decomposition of global 2-SLE}
\label{sec::global2_deomposition}

\begin{definition}\label{def::Q2}
Fix $\kappa\in (4,8)$ and a quad $(\Omega; x_1, x_2, x_3, x_4)$. 
We construct the probability measure  $\QQ_2=\QQ_2(\Omega; x_1, x_2, x_3, x_4)$ on $X_0(\Omega; x_1, x_2, x_3, x_4)$ as follows. 
We fix a M\"obius transformation $\psi$ from $\Omega$ onto $\HH$ such that $\psi(x_1)=\infty$ and we denote 
\[y_0=\psi(x_2),\quad y_1=\psi(x_3),\quad y_2=\psi(x_4).\]
\begin{itemize}
\item 
We first sample a random point $u\in (y_2,\infty)$, the law of which has density 
\begin{equation}\label{eqn::Q2density}
r(u)=\frac{\prod_{j=0}^2(u-y_j)^{-4/\kappa}}{\int_{y_2}^{\infty}\prod_{j=0}^2(v-y_j)^{-4/\kappa}\ud v}.
\end{equation}
\item Given $u$, suppose $\gamma_1$ is $\SLE_{\kappa}(2,2,-4)$ in $\HH$ from $y_0$ to $\infty$ with force points $(y_1, y_2, u)$. Note that almost surely, $\gamma_1$ does not hit $[y_1,u]$ {by~\cite[Lemma 15]{DubedatSLEDuality}}. Denote by $\tau_1$ the hitting time of $\gamma_1$ at $(u,\infty)$.
\item Given $(u, \gamma_1)$, denote by $\LC_{1}$ the connected component of $\HH\setminus\gamma_1$ having $(y_1,y_2)$ on the boundary. Suppose $\gamma_2$ is $\SLE_{\kappa}(2, -4)$ in $\LC_{1}$ from $y_1$ to $\gamma_1(\tau_1)$ with force points $(y_2, u)$, up to the first hitting time of $(y_2,\infty)$. Denote by $\tau_2$ the hitting time of $\gamma_2$ at $(y_2,\infty)$. 
Note that 
$\gamma_2[0,\tau_2]$ does not hit $[y_2,u)\cup(u,\infty)$ and it hits $(y_2, \infty)$ at $u$ {by~\cite[Lemma 15]{DubedatSLEDuality}}.
Denote by $\LC_{2}$ the connected component of $\HH\setminus\gamma_2[0,\tau_2]$ having $y_2$ on the boundary. Define $\gamma_2$ to be the concatenation of $\gamma_2[0,\tau_2]$ and a conditional $\SLE_\kappa$ in $\LC_{2}$ from $u$ to $y_2$. See Figure~\ref{fig::Q2}. 
\end{itemize}
We define $\QQ_2=\QQ_2(\Omega; x_1, x_2, x_3, x_4)$ to be the law of $(\psi^{-1}(\gamma_1),\psi^{-1}(\gamma_2))$. Note that this definition is independent of the choice of $\psi$ once we assume $\psi(x_1)=\infty$. 
\end{definition}

\begin{figure}[ht!]
\begin{center}
\includegraphics[width=0.6\textwidth]{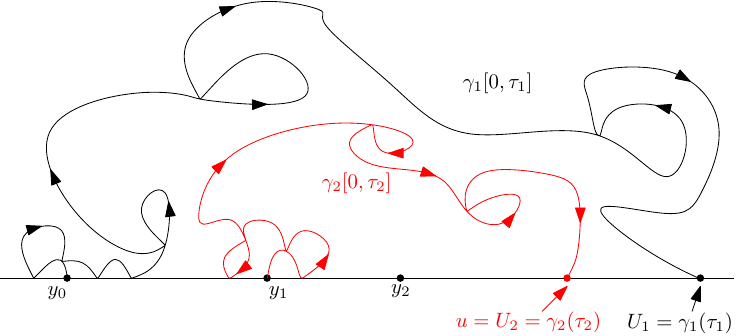}
\end{center}
\caption{\label{fig::Q2} {An illustration of $(\gamma_1,\gamma_2)$ in Definition~\ref{def::Q2}. The black curve is $\gamma_1$ and the red curve is $\gamma_2$. The hitting point of $\gamma_1$ at $(y_2,\infty)$ is $U_1$ and the hitting point of $\gamma_2$ at $(y_2,\infty)$ is $u=U_2$}. }
\end{figure}

\begin{proposition}\label{prop::2SLE_conditionallaw}
The measure $\QQ_2$ in Definition~\ref{def::Q2} is the global $2$-$\SLE_{\kappa}$ on $X_0(\Omega; x_1, x_2, x_3, x_4)$. 
\end{proposition}

We will prove Proposition~\ref{prop::2SLE_conditionallaw} in Section~\ref{subsec::2SLE_conditionallaw} and prove Theorem~\ref{thm::2SLE_decomposition} and Corollary~\ref{cor::2SLE_quad} in Section~\ref{subsec::2SLE_discussion}. 

\subsection{Proof of Proposition~\ref{prop::2SLE_conditionallaw}}
\label{subsec::2SLE_conditionallaw}

We will derive the marginal law of $\gamma_1$ under $\QQ_2$ in Lemma~\ref{lem::eta1underQ1}.
We will derive the conditional law of $\gamma_2$ given $\gamma_1$ under $\QQ_2$ in Lemma~\ref{lem::eta21underQ1}.
{We will recall the characterisation of global $2$-$\SLE$ (see~\cite[Proposition~6.10]{WuHyperSLE}) in Lemma~\ref{lem::characterisation}.}
These three lemmas complete the proof of Proposition~\ref{prop::2SLE_conditionallaw}. 

From conformal invariance, we may assume $\Omega=\HH$. 
To simplify notations, we denote for $y_0<y_1<y_2$ and $u\in\R$,  
\begin{align}
f(y_1, y_2; u):=&(y_2-y_1)^{2/\kappa}\times |u-y_1|^{-4/\kappa}|u-y_2|^{-4/\kappa}; \label{eqn::integrand_21}\\
f(y_0, y_1, y_2; u):=&\prod_{0\leq i<j\leq 2}(y_{j}-y_{i})^{2/\kappa} \times\prod_{0\le j\le 2}|u-y_j|^{-4/\kappa}.\label{eqn::integrand_31}
\end{align}

Note that $f(y_1,y_2;u)$ is the partition function related to $\SLE_\kappa(2,-4)$ from $y_1$ to $\infty$ with force points $(y_2,u)$ and $f(y_0, y_1,y_2;u)$ is the partition function related to $\SLE_\kappa(2, 2,-4)$ from $y_0$ to $\infty$ with force points $(y_1,y_2,u)$.
We start with an auxiliary lemma. Recall that $\kappa\in (4,8)$ and $h=(6-\kappa)/(2\kappa)$. 
\begin{lemma}\label{lem::mart1}
Fix $y_0<y_1<y_2<u$.  
Suppose $\gamma$ is $\SLE_\kappa$ from $y_0$ to $\infty$. 
Denote by $W$ the driving function of $\gamma$ and by $(g_t, t\ge 0)$ the corresponding conformal maps.
We define $\tau$ to be the first hitting time of $\gamma$ at $(y_1, \infty)$. 
Define, for $t<\tau$, 
\begin{equation}\label{eqn::mart1}
M_t=M_t(u):=g_{t}'(y_1)^{h}g_{t}'(y_2)^{h}g_t'(u)f(W_t,g_t(y_1),g_t(y_2);g_t(u)).
\end{equation}
Then, we have the following two properties.
\begin{itemize}
\item
The process $\{M_{t\wedge\tau}\}_{t\ge 0}$ is a uniformly integrable martingale for $\gamma$. 
\begin{itemize}
\item On the event $\{\gamma(\tau)\in (u,\infty)\}$, we have 
\begin{equation}\label{eqn::mart1limit}
\lim_{t\to \tau}M_t=\xi_{\gamma}'(y_1)^{h}\xi_\gamma'(y_2)^{h}\xi_\gamma'(u)f(\xi_\gamma(y_1),\xi_\gamma(y_2);\xi_\gamma(u)), 
\end{equation}
where we denote by $\LC_{\gamma}$ the connected component of $\HH\setminus\gamma$ having $u$ on the boundary and we choose any conformal map $\xi_\gamma$ from $\LC_{\gamma}$ onto $\HH$ such that $\xi_\gamma(\gamma(\tau))=\infty$.
\item On the event $\{\gamma(\tau)\in (y_1, u)\}$, we have 
\begin{equation}\label{eqn::mart1limit2}
\lim_{t\to\tau}M_t=0.
\end{equation}
\end{itemize}
\item
Suppose $\gamma^u$ is an $\SLE_\kappa(2,2,-4)$ curve from $y_0$ to $\infty$ with force points $(y_1,y_2, u)$. The law of $\gamma$ weighted by $\left\{M_{t\wedge\tau}\right\}_{t\ge 0}$ is the same as the law of $\gamma^u$ up to the first hitting time at $(u,\infty)$.
\end{itemize}
\end{lemma}
\begin{proof}
From~\cite[Theorem 6]{SchrammWilsonSLECoordinatechanges}, $M_t$ in~\eqref{eqn::mart1} is a local martingale for $\gamma$ and the law of $\gamma$ weighted by $\left\{M_{t\wedge\tau}\right\}_{t\ge 0}$ is the same as the law of $\gamma^u$ in the second item up to $\tau$. It remains to show that $\{M_{t\wedge\tau}\}_{t\ge 0}$ is uniformly integrable and to derive the terminal values~\eqref{eqn::mart1limit} and~\eqref{eqn::mart1limit2}. 

We decompose $M_t$ in the following way: for $t<\tau$, 
\begin{align}\label{eqn::mart_decomposition}
M_t=&
\left(\underbrace{\frac{g_t'(y_1)g_t'(y_2)}{(g_t(y_2)-g_t(y_1))^2}}_{Z_1(t)}\right)^{-1/\kappa}
\times \left(\underbrace{\frac{g_t'(y_1)g_t'(u)}{(g_t(u)-g_t(y_1))^2}}_{Z_2(t)}\right)^{2/\kappa}
\times \left(\underbrace{\frac{g_t'(y_2)g_t'(u)}{(g_t(u)-g_t(y_2))^2}}_{Z_3(t)}\right)^{2/\kappa}\notag\\
&\times\left(\underbrace{\frac{g_t'(u)^2}{g_t'(y_1)g_t'(y_2)}}_{Z_4(t)}\right)^{(\kappa-4)/(2\kappa)}
\times \left(\underbrace{\frac{(g_t(y_1)-W_t)(g_t(y_2)-W_t)}{(g_t(u)-W_t)^2}}_{Z_5(t)}\right)^{2/\kappa}. 
\end{align}

We first derive the terminal value~\eqref{eqn::mart1limit}. We suppose $\gamma(\tau)\in (u,\infty)$. In this case, $\gamma$ has a positive distance from the interval $[y_1, u]$. 
Let us treat the five terms $Z_1, Z_2, Z_3, Z_4, Z_5$ one by one. 
\begin{itemize}
\item The quantity $Z_1(t)$ is the same as the Poisson kernel for the connected component of $\HH\setminus\gamma[0,t]$ having $y_1, y_2$ on the boundary. As $t\to \tau$, it converges to the Poisson kernel $H(\LC_{\gamma}; y_1, y_2)$. We have similar situation for $Z_2$ and $Z_3$. Thus, 
\[\lim_{t\to \tau}Z_1(t)=
\frac{\xi_{\gamma}'(y_1)\xi_{\gamma}'(y_2)}{(\xi_{\gamma}(y_2)-\xi_{\gamma}(y_1))^2},\quad \lim_{t\to \tau}Z_2(t)=
\frac{\xi_{\gamma}'(y_1)\xi_{\gamma}'(u)}{(\xi_{\gamma}(u)-\xi_{\gamma}(y_1))^2},\quad \lim_{t\to \tau}Z_3(t)=
\frac{\xi_{\gamma}'(y_2)\xi_{\gamma}'(u)}{(\xi_{\gamma}(u)-\xi_{\gamma}(y_2))^2}. \]
\item 
For $Z_4(t)$, 
from Lemma~\ref{lem::technical}, we have
\begin{equation}\label{eqn::mart_decomposition_aux}
\lim_{t\to \tau}Z_4(t)
=\frac{\xi'_{\gamma}(u)^2}{\xi'_{\gamma}(y_1)\xi'_{\gamma}(y_2)}.
\end{equation}
\item For $Z_5$, from Lemma~\ref{lem::technical},  we have
$\lim_{t\to \tau}Z_5(t)=1$.
\end{itemize}
Combining these observations, we obtain~\eqref{eqn::mart1limit} almost surely on $\{\gamma(\tau)\in (u,\infty)\}$. 

Next, we show that $\{M_{t\wedge\tau}\}_{t\ge 0}$ is uniformly integrable.
Let us come back to the decomposition~\eqref{eqn::mart_decomposition}. 
For $n\ge 1$, denote by $T_n$ the hitting time of $\gamma$ at $(1/n)$-neighborhood of $[y_1, u]$. 
\begin{itemize}
\item For $Z_1, Z_2, Z_3$, from the monotonicity of the Poisson kernel Lemma~\ref{lem::Poisson_mono}, we have
\[Z_1(t)\le (y_2-y_1)^{-2},\quad Z_2(t)\le (u-y_1)^{-2},\quad Z_3(t)\le (u-y_2)^{-2}\quad \text{for all }t\le \tau.\]
Denote by $D_n$ the $(1/n)$-neighborhood of $[y_1, u]$ on $\overline\HH$. 
From the monotonicity of the Poisson kernel again,  for all $t\le T_n\wedge \tau$, we have
\[ Z_1(t)\ge H(D_n; y_1,y_2),\quad Z_2(t)\ge H(D_n; y_1, u),\quad Z_1(t)\ge H(D_n; y_2, u).\]
In particular, there exists a constant $C=C_n$, such that $Z_1(t),Z_2(t)$ and $Z_3(t)$ are bounded from below by $C$ for all $t\le T_n\wedge \tau$.

\item For $Z_4$,  it is increasing in $t$, because 
\[\partial_t\log Z_4(t)=\frac{-4}{(g_t(u)-W_t)^2}+\frac{2}{(g_t(y_1)-W_t)^2}+\frac{2}{(g_t(y_2)-W_t)^2}\ge 0,\]
where we use the fact that $W_t\le g_t(y_1)\le g_t(y_2)\le g_t(u)$. Thus, $ Z_4(t)$ is bounded from above by the right-hand side of~\eqref{eqn::mart_decomposition_aux} for all $t\le \tau$. Therefore, $Z_4(t)$ is bounded for all $t\le T_n\wedge \tau$. 
\item  For $Z_5$, as $W_t\le g_t(y_1)\le g_t(y_2)\le g_t(u)$, we have $Z_5(t)\le 1$ for all $t\le \tau$. 
\end{itemize}
Combining these observations, $M_t$ is bounded from above for all $t\le T_n\wedge \tau$. Denote by $P_n$ the law of $\gamma$ weighted by $M_{T_n\wedge \tau}/M_0$. Then $P_n$ is the same as the law of $\gamma^u$ up to the minimum of the first hitting time at $(1/n)$-neighborhood of $[y_1, u]$ and the first hitting time at $(u,\infty)$. As $\{P_n\}_n$ is compatible in $n$, there exists a probability measure $P_*$ such that under $P_*$, for each $n$, the curve has the law of $\gamma^u$ up to the minimum of the first hitting time at $(1/n)$-neighborhood of $[y_1, u]$  and the first hitting time at $(u,\infty)$. At the same time, we know that $\gamma^u$ has a positive distance from the interval $[y_1, u]$ almost surely. Hence $P_*$ is exactly the law of $\gamma^u$ up to the first hitting time at $(u,\infty)$. This implies that $\{M_{t\wedge\tau}\}_{t\ge 0}$ is uniformly integrable and the law of $\gamma$ weighted by $\{M_{t\wedge\tau}\}_{t\ge 0}$ is the same as the law of $\gamma^u$  up to the first hitting time at $(u,\infty)$. 

Finally, we derive~\eqref{eqn::mart1limit2}. As $\{M_{t\wedge\tau}\}_{t\ge 0}$ is uniformly integrable and the law of $\gamma$ weighted by $\{M_{t\wedge\tau}\}_{t\ge 0}$ is the same as the law of $\gamma^u$  up to the first hitting time at $(u,\infty)$, we have 
\[\E\left[\one_{\{\gamma(\tau)\in (y_1, u)\}}M_{\tau}/M_0\right]=\PP[\gamma^u \text{ hits }(y_1, u) \text{ before hitting }(u,\infty)]=0.\]
This gives~\eqref{eqn::mart1limit2} and completes the proof.
\end{proof}

Next, we will use Lemma~\ref{lem::mart1} to derive the law of $\gamma_1$ under $\QQ_2$. Recall that we assume $\Omega=\HH$ and suppose $x_1<x_2<x_3<x_4$.  We fix any M\"obius transformation $\psi$ such that $\psi(x_1)=\infty$.
Recall that we denote $y_0=\psi(x_2), y_1=\psi(x_3), y_2=\psi(x_4)$. To simplify notations, we denote 
\begin{equation}\label{eqn::notation_simplified_G}
G=G(y_0, y_1, y_2)=\int_{y_2}^{\infty}f(y_0,y_1, y_2; u)\ud u. 
\end{equation}
\begin{lemma}\label{lem::eta1underQ1}
Suppose $\gamma$ is $\SLE_\kappa$ from $y_0$ to $\infty$. On the event $\{\gamma\cap[y_1,y_2]=\emptyset\}$, we denote by $\LC_{\gamma}$ the connected component of $\HH\setminus\gamma$ having $y_1, y_2$ on the boundary. Then the law of $\gamma_1$ under $\QQ_2$ is the same as the law of $\gamma$ weighted by the Poisson kernel to the power $h$: 
\begin{equation}\label{eqn::eta1Q1}
\frac{C_\kappa}{G} \one_{\{\gamma\cap[y_1, y_2]=\emptyset\}}\xi_{\gamma}'(y_1)^{h}\xi_{\gamma}'(y_2)^{h}(\xi_{\gamma}(y_2)-\xi_{\gamma}(y_1))^{-2h}
\end{equation}
for any conformal map $\xi_{\gamma}$ from $\LC_{\gamma}$ onto $\HH$ with $\xi_{\gamma}(y_1)<\xi_{\gamma}(y_2)$. 
\end{lemma}

Note that~\eqref{eqn::eta1Q1} agrees with the partition function of $\SLE_\kappa$ from $\xi_\gamma(y_1)$ to $\xi_\gamma(y_2)$. Thus,  the random curve $(\gamma_1,\gamma_2)$ under $\QQ_2$ can be sampled as follows informally: we first sample $\SLE_\kappa$ from $y_0$ to $\infty$ weighted by the partition function of $\SLE_\kappa$ from $y_1$ to $y_2$ in the remaining domain. Then, we sample the $\SLE_\kappa$ from $y_1$ to $y_2$ in the remaining domain. This implies that Proposition~\ref{prop::2SLE_conditionallaw} holds in the level of partition functions--the Radon-Nikodym derivative of the first curve with respect to $\SLE_\kappa$ equals the partition function of the second curve. 
\begin{proof}
We will compare the law of $\gamma_1$ under $\QQ_2$ with respect to the law of $\gamma\sim\SLE_{\kappa}$ from $y_0$ to $\infty$. 
From the construction of $\QQ_2$, the law of $\gamma_1$ given $u$ is $\SLE_{\kappa}(2,2,-4)$ from $y_0$ to $\infty$ with force points $(y_1, y_2, u)$ whose law is the same as the law of $\gamma$ weighted by the martingale~\eqref{eqn::mart1} up to the first hitting time of $(u,\infty)$. Moreover, after hitting $(u,\infty)$, the random curve $\gamma_1$ evolves as standard $\SLE_\kappa$. 
We only need to derive the law of the curves up to the hitting time of $(u,\infty)$. 
For any bounded continuous function $F$ on curve space, we have
\begin{align*}
&\QQ_2[F(\gamma_1)]=\int_{y_2}^{\infty}r(u)\ud u \E\left[F(\gamma^u)\right]=\int_{y_2}^{\infty}\frac{r(u)\ud u}{M_0(u)} \E\left[F(\gamma)M_{\tau^u}(u)\right],
\end{align*}
where $r(u)$ is defined by~\eqref{eqn::Q2density} and $M_t(u)$ is the martingale in~\eqref{eqn::mart1} and $\tau^u$
is the first hitting time of curves at $(u,\infty)$. 

We denote by $T_{\gamma}$ the hitting time of $\gamma$ at $[y_2,\infty)$ and on the event $\{\gamma\cap [y_1, y_2]=\emptyset\}$, we denote by $\xi_{\gamma}$ any conformal map from the connected component of $\HH\setminus\gamma$ having $y_1, y_2$ on the boundary onto $\HH$. Then we have
\begin{align}
&\QQ_2[F(\gamma_1)]=\int_{y_2}^{\infty}\frac{r(u)\ud u}{M_0} \E\left[F(\gamma)M_{\tau^u}(u)\right]\notag\\
=&\frac{1}{G}\int_{y_2}^{\infty}\ud u\E\left[F(\gamma)\one_{\{\gamma(T_{\gamma})\in (u,\infty)\}}\xi_\gamma'(y_1)^{h}\xi_\gamma'(y_2)^{h}\xi_\gamma'(u)f(\xi_\gamma(y_1),\xi_\gamma(y_2);\xi_\gamma(u))\right]\notag\\
=&\frac{1}{G}\E\left[F(\gamma)\one_{\{\gamma\cap [y_1, y_2]=\emptyset\}}\xi_\gamma'(y_1)^{h}\xi_\gamma'(y_2)^{h}(\xi_{\gamma}(y_2)-\xi_{\gamma}(y_1))^{2/\kappa}
\int_{y_2}^{\gamma(T_{\gamma})}\frac{\xi_\gamma'(u)\ud u}{(\xi_{\gamma}(u)-\xi_{\gamma}(y_1))^{4/\kappa}(\xi_{\gamma}(u)-\xi_{\gamma}(y_2))^{4/\kappa}}\right]\notag\\
=&\frac{1}{G}\E\left[F(\gamma)\one_{\{\gamma\cap [y_1, y_2]=\emptyset\}}\xi_\gamma'(y_1)^{h}\xi_\gamma'(y_2)^{h}(\xi_{\gamma}(y_2)-\xi_{\gamma}(y_1))^{2/\kappa}
\int_{\xi_{\gamma}(y_2)}^{\infty}\frac{\ud w}{(w-\xi_{\gamma}(y_1))^{4/\kappa}(w-\xi_{\gamma}(y_2))^{4/\kappa}}\right]\notag\\
=&\frac{1}{G}\E\left[F(\gamma)\one_{\{\gamma\cap [y_1, y_2]=\emptyset\}}\xi_\gamma'(y_1)^{h}\xi_\gamma'(y_2)^{h}(\xi_{\gamma}(y_2)-\xi_{\gamma}(y_1))^{-2h}
\int_{0}^{1}\frac{\ud s}{s^{4/\kappa}(1-s)^{2-8/\kappa}}\right]\tag{set $s=\frac{w-\xi_{\gamma}(y_2)}{w-\xi_{\gamma}(y_1)}$}\\
=&\frac{C_{\kappa}}{G}\E[F(\gamma)\one_{\{\gamma\cap[y_1,y_2]=\emptyset\}}\xi_\gamma'(y_1)^{h}\xi_\gamma'(y_2)^{h}(\xi_\gamma(y_2)-\xi_\gamma(y_1))^{-2h}]. \tag{due to~\eqref{eqn::constC_Beta}}
\end{align}
 This completes the proof.
\end{proof}
Next, we will derive the conditional law of $\gamma_2$ given $\gamma_1$ under $\QQ_2$.

\begin{lemma}\label{lem::eta21underQ1}
Under $\QQ_2$, the conditional law of $\gamma_2$ given $\gamma_1$ equals the $\SLE_\kappa$ from $y_1$ to $y_2$ in the connected component of $\HH\setminus\gamma_1$ having $y_1, y_2$ on the boundary.
\end{lemma}
\begin{proof}
We have the following two observations.
\begin{itemize}
\item We compare the law of $\gamma_1$ with the law of $\gamma\sim\SLE_{\kappa}$ from $y_0$ to $\infty$ and denote by $\xi_{\gamma}$ any conformal map from the connected component of $\HH\setminus\gamma$ having $y_1, y_2$ on the boundary onto $\HH$. 
The law of $\gamma_1$ given $u$ is $\SLE_{\kappa}(2,2,-4)$ from $y_0$ to $\infty$ with force points $(y_1, y_2, u)$ which is the same as the law of $\gamma$ weighted by the martingale~\eqref{eqn::mart1}. After hitting $(u,\infty)$, the random curve $\gamma_1$ evolves as standard $\SLE_\kappa$.
\item Denote by $\xi_{\gamma_1}$ any conformal map from the connected component of $\HH\setminus\gamma_1$ having $y_1, y_2$ on the boundary onto $\HH$. By the third item in Definition~\ref{def::Q2}, the conditional law of $\xi_{\gamma_1}(\gamma_2)$ given $\gamma_1$ and $u$ equals $\SLE_{\kappa}(2,-4)$ starting from $\xi_{\gamma_1}(y_1)$ with force points $(\xi_{\gamma_1}(y_2), \xi_{\gamma_1}(u))$. 
\end{itemize}
Combining these two observations, for any bounded continuous functions $F_1, F_2$ on curve space, we have 
\begin{align}\label{eqn::eta21underQ1_aux}
&\QQ_2[F_1(\gamma_1)F_2(\xi_{\gamma_1}(\gamma_2))]\notag\\
=&\QQ_2\left[\QQ_2\left[F_1(\gamma_1)\QQ_2[F_2(\xi_{\gamma_1}(\gamma_2))\cond \gamma_1, u]\cond u\right]\right]\notag\\
=&\frac{1}{G}\E\left[F_1(\gamma)\one_{\{\gamma\cap[y_1,y_2]=\emptyset\}}\xi'_{\gamma}(y_1)^{h}\xi'_{\gamma}(y_2)^{h}\underbrace{\int_{\xi_{\gamma}(y_2)}^{\infty}\QQ_2[F_2(\xi_{\gamma}(\gamma_2))\cond \gamma, w]f(\xi_{\gamma}(y_1),\xi_{\gamma}(y_2);w)\ud w}_{Z}\right],
\end{align}
where $w=\xi_\gamma(u)$. Then we consider the term $\QQ_2[F_2(\xi_{\gamma}(\gamma_2))\cond \gamma, w]$ in the integral $Z$. We already see that, given $\gamma$ and $w$, the conditional law of $\xi_{\gamma}(\gamma_2)$ equals $\SLE_{\kappa}(2,-4)$ starting from $\hat{y}_1:=\xi_{\gamma}(y_1)$ with force points $(\hat{y}_2:=\xi_{\gamma}(y_2), w)$. Let us compare its law with the law of $\hat{\gamma}\sim\SLE_{\kappa}$ from $\hat{y}_1$ to $\hat{y}_2$. We parameterize $\hat\gamma$ by half-plane capacity. Denote by $(\hat W_t, t\ge 0)$ its driving function and by $(\hat g_t, t\ge 0)$ the corresponding conformal maps for $\hat{\gamma}$. 
Note that $\hat{\gamma}$ is the same as $\SLE_{\kappa}(\kappa-6)$ from $\hat{y}_1$ to $\infty$ up to the swallowing time of $\hat{y}_2$. 
Define 
\[N_t^w=(g_t(\hat{y}_2)-\hat W_t)^{2h} \hat g_t'(w)f(\hat W_t,\hat g_t(\hat{y}_2);\hat g_t(w)).\]
Then $\{N^w_t\}_{t\ge 0}$ is a local martingale for $\hat{\gamma}$ and the law of $\hat{\gamma}$ weighted by $\{N^w_t\}_{t\ge 0}$ is the same as $\SLE_{\kappa}(2,-4)$ from $\hat{y}_1$ with force points $(\hat{y}_2, w)$ up to proper stopping time. 
Denote by $T_n$ the first time that the curve hits  the $(1/n)$-neighborhood of $[\hat{y}_2, \infty)$. Then we have 
\begin{align*}
&\int_{\hat{y}_2}^{\infty}\QQ_2[F_2(\xi_{\gamma}(\gamma_2)[0,T_n])\cond \gamma, w]f(\hat{y}_1, \hat{y}_2; w)\ud w\\
=&\int_{\hat{y}_2}^{\infty}\frac{1}{N_0^w}\E[F_2(\hat{\gamma}[0,T_n])N_{T_n}^w]f(\hat{y}_1, \hat{y}_2; w)\ud w\\
=&(\hat{y}_2-\hat{y}_1)^{-2h}\E\left[F_2(\hat{\gamma}[0,T_n])(\hat g_{T_n}(\hat{y}_2)-\hat W_{T_n})^{2h+2/\kappa} \int_{\hat{y}_2}^{\infty}\frac{\hat g_{T_n}'(w)\ud w}{(\hat g_{T_n}(w)-\hat W_{T_n})^{4/\kappa}(\hat g_{T_n}(w)-\hat g_{T_n}(\hat{y}_2))^{4/\kappa}}\right]\\
=&(\hat{y}_2-\hat{y}_1)^{-2h}\E\left[F_2(\hat{\gamma}[0,T_n])(\hat g_{T_n}(\hat{y}_2)-\hat W_{T_n})^{2h+2/\kappa} \int_{\hat g_{T_n}(\hat{y}_2)}^{\infty}\frac{\ud v}{(v-\hat W_{T_n})^{4/\kappa}(v-\hat g_{T_n}(\hat{y}_2))^{4/\kappa}}\right]\\
=&(\hat{y}_2-\hat{y}_1)^{-2h}\E\left[F_2(\hat{\gamma}[0,T_n])
\int_0^1\frac{\ud s}{s^{4/\kappa}(1-s)^{2-8/\kappa}}\right]\tag{set $s=\frac{v-\hat W_{T_n}}{v-\hat g_{T_n}(\hat{y}_2)}$}\\
=&C_{\kappa}(\hat{y}_2-\hat{y}_1)^{-2h}\E\left[F_2(\hat{\gamma}[0,T_n])
\right]. \tag{due to~\eqref{eqn::constC_Beta}}
\end{align*}
Plugging into the integral $Z$ in~\eqref{eqn::eta21underQ1_aux} and letting $n\to\infty$ (recall that $\gamma_2$ is a curve stopped when it hits $[y_2, \infty)$), we have
\[Z=C_{\kappa}(\hat{y}_2-\hat{y}_1)^{-2h}\E\left[F_2(\hat{\gamma}[0,\hat{T}])
\right], \]
where $\hat{T}$ is the hitting time of $\hat{\gamma}$ at $[\hat{y}_2,\infty)$. 
Plugging into~\eqref{eqn::eta21underQ1_aux}, we have
\begin{align}
\QQ_2[F_1(\gamma_1)F_2(\xi_{\gamma_1}(\gamma_2))]&=\frac{C_{\kappa}}{G}
\E\left[F_1(\gamma)\one_{\{\gamma\cap[y_1,y_2]=\emptyset\}}\xi'_{\gamma}(y_1)^{h}\xi'_{\gamma}(y_2)^{h}(\xi_{\gamma}(y_2)-\xi_{\gamma}(y_1))^{-2h}\E\left[F_2(\hat{\gamma}[0,\hat{T}])
\right]\right]\notag\\
&=\QQ_2\left[F_1(\gamma_1)\E\left[F_2(\hat{\gamma}[0,\hat{T}])
\right]\right],\tag{due to Lemma~\ref{lem::eta1underQ1}}
\end{align}
as desired.
\end{proof}

Next, we recall a characterisation of global $2$-$\SLE$. 
\begin{lemma}{\cite[Proposition~6.10]{WuHyperSLE}}\label{lem::characterisation}
Fix $\kappa\in (0,8)$ and a quad $(\Omega; x_1, x_2, x_3, x_4)$. Suppose $(\eta_1, \eta_2)$ is global 2-$\SLE_{\kappa}$ as in Definition~\ref{def::2SLE}.  
\begin{itemize}
\item Given $\eta_1$, we denote by $\Omega^R$ the connected component of $\Omega\setminus\eta_1$ having $(x_3x_4)$ on the boundary.  The law of global 2-$\SLE_{\kappa}$ can be characterized as follows:
\begin{itemize}
\item the marginal of $\eta_1$ is $\SLE_{\kappa}$ in $\Omega$ connecting $x_1$ and $x_2$ weighted by $\one_{\{\eta_1\cap(x_3x_4)=\emptyset\}}H(\Omega^R; x_3, x_4)^h$; 
\item the conditional law of $\eta_2$ given $\eta_1$ is $\SLE_{\kappa}$ in $\Omega^R$ connecting $x_3$ and $x_4$.
\end{itemize}
\item Given $\eta_2$, we denote by $\Omega^L$ the connected component of $\Omega\setminus\eta_2$ having $(x_1x_2)$ on the boundary. The law of global $2$-$\SLE_{\kappa}$ can be characterized as follows:
\begin{itemize}
\item the marginal of $\eta_2$ is $\SLE_{\kappa}$ in $\Omega$ connecting $x_3$ and $x_4$ weighted by $\one_{\{\eta_2\cap(x_1x_2)=\emptyset\}}H(\Omega^L; x_1, x_2)^h$; 
\item the conditional law of $\eta_1$ given $\eta_2$ is $\SLE_{\kappa}$ in $\Omega^L$ connecting $x_1$ and $x_2$.
\end{itemize}
\end{itemize}
\end{lemma}
\begin{proof}[Proof of Proposition~\ref{prop::2SLE_conditionallaw}]
Combining Lemmas~\ref{lem::eta1underQ1},~\ref{lem::eta21underQ1} and Lemma~\ref{lem::characterisation} together, we see that $\QQ_2$ is the same as global 2-$\SLE_{\kappa}$ as desired. 
\end{proof}
%

\subsection{Proof of Theorem~\ref{thm::2SLE_decomposition} and Corollary~\ref{cor::2SLE_quad}}
\label{subsec::2SLE_discussion}
\begin{proof}[Proof of Theorem~\ref{thm::2SLE_decomposition} and Corollary~\ref{cor::2SLE_quad}]
The conclusion of Theorem~\ref{thm::2SLE_decomposition} follows from the construction in Definition~\ref{def::Q2} and Proposition~\ref{prop::2SLE_conditionallaw}. 
It remains to show Corollary~\ref{cor::2SLE_quad}. 
Recall that we fix $\kappa\in (4,8)$ and a quad $(\Omega; x_1, x_2, x_3, x_4)$ and suppose $(\eta_1, \eta_2)\in X_0(\Omega; x_1, x_2, x_3, x_4)$ is a global $2$-$\SLE_{\kappa}$ such that $\eta_1$ goes from $x_2$ to $x_1$ and $\eta_2$ goes from $x_3$ to $x_4$. Let $\tau_2$ be the first time that $\eta_2$ hits the boundary arc $(x_4x_1)$. Suppose $\psi$ is a conformal map from $\Omega$ onto $\HH$ such that $\psi(x_1)=\infty$. 
Define 
\[\xi(w)=\int_{\psi(x_2)}^{w}\prod_{j=2}^4(z-\psi(x_j))^{-4/\kappa}\ud z, \quad w\in\HH.\]
This is the Schwarz-Christoffel formula that gives the conformal map from $\HH$ onto a quadrilateral such that $(\psi(x_1), \psi(x_2), \psi(x_3), \psi(x_4))$ are mapped to the four corners of the quadrilateral with interior angles given by~\eqref{eqn::2SLE_quad_angles}. Thus, $\phi=\xi\circ\psi$ is the conformal map from $\Omega$ onto a quadrilateral with four corners $(\phi(x_1), \phi(x_2), \phi(x_3), \phi(x_4))$ with interior angles given by~\eqref{eqn::2SLE_quad_angles}. As the law of $\psi(\eta_2(\tau_2))$ is given by the density $r(u)$ in~\eqref{eqn::2SLE_hittingpoint_density} and $\xi'(u)$ is proportional to $r(u)$, for any bounded continuous function $F$, we have 
\begin{align*}
\E[F(\phi(\eta_2(\tau_2)))]=\E[F(\xi(\psi(\eta_2(\tau_2))))]=\int_{\psi(x_4)}^{\infty}F(\xi(u))r(u)\ud u=\frac{1}{|(\phi(x_4)\phi(x_1))|}\int_{\phi(x_4)}^{\phi(x_1)}F(z)\ud z. 
\end{align*}
This implies that $\phi(\eta_2(\tau_2))$ is uniform on $(\phi(x_4)\phi(x_1))$ as desired. 
\end{proof}

Let us consider Corollary~\ref{cor::2SLE_quad} with three specific values of $\kappa$. 
\begin{itemize}
\item When $\kappa\to 8$, the resulting quadrilateral in Corollary~\ref{cor::2SLE_quad} is a rectangle, see Figure~\ref{fig::2SLE_quad}(b). In this case, the conclusion in Corollary~\ref{cor::2SLE_quad} coincides with the one in~\cite[Theorem~1.6]{HanLiuWuUST} where the authors derive the law of the first hitting point of one curve in a pair of $\SLE_8$ in the setup of uniform spanning tree in rectangles with alternating boundary conditions. 
\item When $\kappa=6$, the resulting quadrilateral in Corollary~\ref{cor::2SLE_quad} becomes an equilateral triangle with three corners $(\phi(x_2), \phi(x_3), \phi(x_4))$ and $\phi(x_1)$ is a point on $(\phi(x_4)\phi(x_2))$, see Figure~\ref{fig::2SLE_quad}(c). In this case, the conclusion in Corollary~\ref{cor::2SLE_quad} coincides with known result about $\SLE_6$.  Suppose $\eta$ is an $\SLE_6$ in the equilateral triangle with three corners $(\phi(x_2), \phi(x_3), \phi(x_4))$ starting from $\phi(x_3)$ to $\phi(x_4)$ and let $\tau$ be the first time that $\eta$ hits the boundary arc $(\phi(x_4)\phi(x_2))$. 
\begin{itemize}
\item On the one hand, it is known that $\eta(\tau)$ is uniform on $(\phi(x_4)\phi(x_2))$.  
\item On the other hand, it is known that the marginal law of $\eta_2$ in the pair $(\eta_1, \eta_2)$ is the same as $\SLE_6$ conditioned not to hit the arc $(x_1x_2)$, see Lemma~\ref{lem::characterisation} (in this case, we have $h=0$). In other words, $\phi(\eta_2)$ has the same law as $\eta$ conditioned not to hit $(\phi(x_1)\phi(x_2))$.
\end{itemize}
Combining these two observations, we see that the law of $\phi(\eta_2(\tau))$ is uniform in $(\phi(x_4)\phi(x_1))$. This is the same as the conclusion in Corollary~\ref{cor::2SLE_quad} with $\kappa=6$. 
\item When $\kappa\to 4$, the quadrilateral degenerates. This is consistent with the known results for $\kappa=4$. In this case, global $2$-$\SLE_4$ is a pair of simple curves $(\eta_1, \eta_2)$ such that $\eta_1$ and $\eta_2$ only touch the boundary at the endpoints. In particular, the hitting point $\eta_2(\tau)$ coincides with the point $x_4$. In this case, the pair $(\eta_1, \eta_2)$ can be coupled with GFF as level lines\footnote{See introduction for level lines of GFF in e.g.~\cite{SchrammSheffieldContinuumGFF} and~\cite{WangWuLevellinesGFFI}.}, this can be viewed as a degenerate situation in Theorem~\ref{thm::2SLE_decomposition}. We fix $\lambda=\pi/2$. Suppose $\Gamma$ is a GFF in $\Omega$ with the boundary data: 
\[-\lambda\text{ on }(x_1x_2), \quad \lambda\text{ on }(x_2x_3),\quad 3\lambda\text{ on }(x_3x_4), \quad \lambda\text{ on }(x_4x_1). \]
Let $\eta_1$ be the level line of $\Gamma$ from $x_2$ to $x_1$ and let $\eta_2$ be the level line of $\Gamma-2\lambda$ from $x_3$ to $x_4$. Then $(\eta_1, \eta_2)$ has the law of global $2$-$\SLE_4$. 
\end{itemize}

We end this section with a technical lemma for global $2$-SLE that will be useful in Section~\ref{sec::FKIsing_mono}. 
\begin{lemma} \label{lem::FK_mono_aux1}
Fix $\kappa\in (4,8)$. 
Suppose that $(\eta_1,\eta_2)\in X_0(\Omega;x_1,x_2,x_3,x_4)$ is a global $2$-$\mathrm{SLE}_{\kappa}$ such that $\eta_1$ goes from $x_2$ to $x_1$ and $\eta_2$ goes from $x_3$ to $x_4$. Then we have
\begin{equation}\label{eqn::lem::FK_mono_aux1}
	\mathbb{Q}_2\left[\left(\eta_1\cap\eta_2 \cap \partial \Omega\right)\neq \emptyset\right]=0.
\end{equation}
\end{lemma}
\begin{proof}
From the construction in Definition~\ref{def::Q2}, we have $\eta_1\cap (x_3x_4)=\emptyset$, which justifies that $\eta_1\cap\eta_2\cap(x_3x_4)=\emptyset$. By symmetry, we have $\eta_1\cap\eta_2\cap(x_1x_2)=\emptyset$. It remains to check $\eta_1\cap\eta_2\cap(x_2x_3)$ and $\eta_1\cap\eta_2\cap(x_4x_1)$. For $\eta_2$, let $\tau$ be the hitting time of $\eta_2$ at $(x_4x_1)$ and let $\sigma$ be the last hitting time of $\eta_2$ at $(x_1x_2)$ before $\tau$. From the construction in Definition~\ref{def::Q2}, the left boundary of $\eta_2[\sigma, \tau]$ hits $\partial\Omega$ only at $\eta_2(\sigma)$ and $\eta_2(\tau)$. Thus, the only possible point for $\eta_1\cap\eta_2\cap(x_2x_3)$ is $\eta_2(\sigma)$ and the only possible point for $\eta_1\cap\eta_2\cap(x_4x_1)$ is $\eta_2(\tau)$. As the conditional law of $\eta_1$ given $\eta_2$ is $\SLE_{\kappa}$ and $\SLE_{\kappa}$ does not hit any given boundary point except the endpoints, we have 
\[\QQ_2[\eta_1\text{ hits }\eta_2(\tau)\text{ or }\eta_2(\sigma)\cond \eta_2]=0.\]
Thus $\eta_1\cap\eta_2\cap(x_1x_2)=\emptyset$ and $\eta_1\cap\eta_2\cap(x_4x_1)=\emptyset$ almost surely. This completes the proof. 
\end{proof}


\section{Proof of Proposition~\ref{prop::global2_mono}}
\label{sec::global2_mono}
The goal of this section is to prove Proposition~\ref{prop::global2_mono}. To this end, we first give an auxiliary conclusion about global 2-$\SLE_{\kappa}$ in Section~\ref{subsec::global2_mono_aux} and then prove Proposition~\ref{prop::global2_mono} in Section~\ref{subsec::global2_mono}. 

\subsection{Estimates on global 2-SLE}
\label{subsec::global2_mono_aux} 

Fix $\kappa\in (4,8)$ and $y_0<y_1<y_2$. 
Recall that $f(y_1, y_2; u)$ and $f(y_0, y_1, y_2; u)$ are defined in~\eqref{eqn::integrand_21} and~\eqref{eqn::integrand_31} and we further define: for $y_0<y_1<y_2$ and $u_1, u_2\in\R$, 
\begin{equation}\label{eqn::integrand_32}
f(y_0, y_1, y_2; u_1, u_2)=\prod_{0\leq i<j\leq 2}(y_{j}-y_{i})^{2/\kappa} \times\prod_{0\le j\le 2}|u_1-y_j|^{-4/\kappa}\times\prod_{0\le j\le 2}|u_2-y_j|^{-4/\kappa}\times |u_2-u_1|^{8/\kappa}. 
\end{equation}

We will prove the following identity, Proposition~\ref{prop::nonintersecting}, for the global $2$-SLE$_{\kappa}$ in this section. We use the same notations as in Definition~\ref{def::Q2} and denote $y_0=\psi(x_2), y_1=\psi(x_3), y_2=\psi(x_4)$ as before. 
We sample the triple $(u, \gamma_1, \gamma_2)$ as in Definition~\ref{def::Q2}, and set $U_2=u$, i.e. $U_2$ is a random point with density~\eqref{eqn::Q2density}.
Recall that $C_{\kappa}=B(1-4/\kappa, 8/\kappa-1)$ and $G=G(y_0, y_1, y_2)$ are defined in~\eqref{eqn::constant_def}, \eqref{eqn::constC_Beta} and~\eqref{eqn::notation_simplified_G}. 

\begin{proposition}\label{prop::nonintersecting}
Define $\tau_1$ (resp. $\tau_2$) to be the first hitting time of $\gamma_1$ (resp. $\gamma_2$) at $(y_2,\infty)$ and define $\sigma_1$ (resp. $\sigma_2$) to be the last hitting time of $\gamma_1$ (resp. $\gamma_2$) at $(y_0,y_1)$ before $\tau_1$ (resp. $\tau_2$). See Figure~\ref{fig::U1U2B1B2}. 
We denote $U_1:=\gamma_1(\tau_1)$ and $B_1:=\gamma_1(\sigma_1)$ and $B_2:=\gamma_2(\sigma_2)$ and we note that $U_2=\gamma_2(\tau_2)$. 
Fix $b, u_1, u_2$ such that $y_0<b<y_1<y_2<u_2<u_1$. Then, we have
\begin{align}\label{eqn::hitting1}
&\QQ_2[\gamma_1\cap\gamma_2=\emptyset, B_1\in (y_0,b), B_2\in (b,y_1), U_2\in (y_2,u_2), U_1\in (u_1,\infty)]\notag\\
=&\frac{1}{C_{\kappa}G}\int^{u_2}_{y_2}\ud v_2\int^{\infty}_{u_1}\ud v_1 f(y_0,y_1,y_2;v_1,v_2)\times (b-y_0)^{1-4/\kappa}(y_1-b)^{1-4/\kappa}(y_2-b)^{1-4/\kappa}(v_2-b)^{8/\kappa-2}(v_1-b)^{8/\kappa-2}.
\end{align}
\end{proposition} 
\begin{figure}[ht!]
\begin{center}
\includegraphics[width=0.6\textwidth]{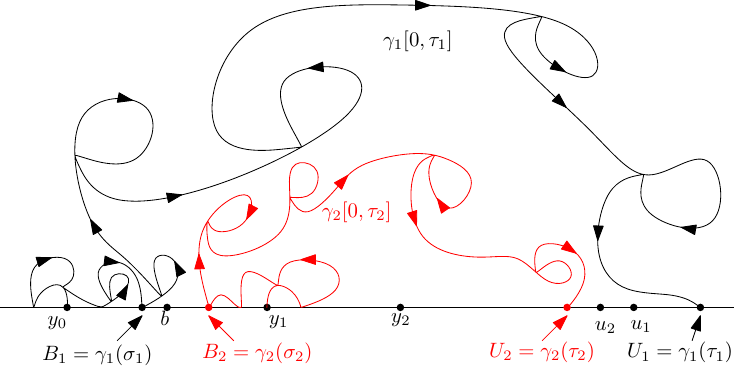}
\end{center}
\caption{\label{fig::U1U2B1B2} An illustration of Proposition~\ref{prop::nonintersecting}. The black curve is $\gamma_1$ and the red curve is $\gamma_2$. The last hitting point of $\gamma_1$ at $(y_0,y_1)$ is $B_1\in (y_0,b)$ and the last hitting point of $\gamma_2$ at $(y_0,y_1)$ is $B_2\in (b,y_1)$. The first hitting point of $\gamma_1$ at $(y_2,\infty)$ is $U_1\in (u_1,\infty)$ and the first hitting point of $\gamma_2$ at $(y_2,\infty)$ is $U_2\in (y_2,u_2)$. Moreover, we require $\gamma_1\cap\gamma_2=\emptyset$.}
\end{figure}

To prove Proposition~\ref{prop::nonintersecting}, we begin with two auxiliary lemmas.
\begin{lemma}\label{lem::aux1}
Fix $\kappa\in(4,8)$ and fix $y_0<b<y_1<y_2<v_2<u_1$. Suppose $\gamma$ is $\SLE_{\kappa}$ from $y_0$ to $\infty$. Denote by $W$ the driving function of $\gamma$ and by $(g_t, t\ge 0)$ the corresponding conformal maps. Let $\tau$ be the first hitting time of $\gamma$ at $(b,\infty)$ and define, for $t<\tau$, 
\begin{align}
R_t=R_t(v_2):=&g'_t(y_1)^hg'_t(y_2)^h(g_t(b)-W_t)^{1-4/\kappa}(g_t(y_1)-g_t(b))^{1-4/\kappa}(g_t(y_2)-g_t(b))^{1-4/\kappa} (g_t(v_2)-g_t(b))^{8/\kappa-2}\notag\\
&\times g'_t(v_2)\times\int^{\infty}_{g_t(u_1)}\ud v_1 f(W_t,g_t(y_1),g_t(y_2);g_t(v_2),v_1)\times (v_1-g_t(b))^{8/\kappa-2}.
\end{align}
We define $E=\{\gamma(\tau)\in (u_1, \infty)\}$, and on the event $E$, we define $\LC_\gamma$ to be the connected component of $\HH\setminus\gamma[0,\tau]$ having $v_2$ on the boundary. Define $\xi_\gamma$ to be the conformal map from $\LC_\gamma$ onto $\HH$ such that $\xi_\gamma(\gamma(\tau))=\infty$. Then, $\{R_{t\wedge\tau}\}_{t\ge 0}$ is a uniformly integrable martingale and 
\begin{align}\label{eqn::terminalnonintersecting}
\lim_{t\to\tau}R_t=R_\tau=&\one_E\times C_{\kappa}\xi'_\gamma(y_1)^h\xi'_\gamma(y_2)^h\xi'_\gamma(v_2)f(\xi_\gamma(y_1),\xi_\gamma(y_2);\xi_\gamma(v_2))\notag\\
&\times (\xi_\gamma(y_1)-\xi_\gamma(b))^{1-4/\kappa}(\xi_\gamma(y_2)-\xi_\gamma(b))^{1-4/\kappa}(\xi_\gamma(v_2)-\xi_\gamma(b))^{8/\kappa-2}.
\end{align}
\end{lemma}
\begin{proof}
For $y_2<v_2<v_1$, we define, for $t<\tau$,
\begin{align*}
N_t(v_1,v_2):=&g'_t(y_1)^hg'_t(y_2)^hg'_t(v_1)g'_t(v_2)(g_t(b)-W_t)^{1-4/\kappa}(g_t(y_1)-g_t(b))^{1-4/\kappa}(g_t(y_2)-g_t(b))^{1-4/\kappa}\\
&\times f(W_t,g_t(y_1),g_t(y_2);g_t(v_2),g_t(v_1))\times (g_t(v_1)-g_t(b))^{8/\kappa-2}(g_t(v_2)-g_t(b))^{8/\kappa-2}.
\end{align*}
From~\cite[Theorem 6]{SchrammWilsonSLECoordinatechanges}, $N_t(v_1,v_2)$ is a local martingale for $\gamma$. 
Moreover, before hitting $[b,+\infty)$, the law of $\gamma$ weighted by $\{N_t(v_1,v_2)\}_{t\ge 0}$ equals $\SLE_\kappa(\kappa-4,2,2,-4,-4)$ from $y_0$ to $\infty$ with force points $(b,y_1,y_2,v_2,v_1)$. 
For every $M>0$, denote by $S_M:=\inf\{t\ge 0: N_t(v_1,v_2)\ge M\}$. For every $\eps>0$, denote by $\tau_\eps$ the hitting time of $\gamma$ at the $\eps$-neighbourhood of $(b,+\infty)$. Then, for every $0<t_1<t_2$ and a positive measurable function $F$ on curve space, we have 
\[\E[N_{t_2\wedge\tau_\eps\wedge S_M}F(\gamma[t_1\wedge\tau_\eps\wedge S_M])]=\E[N_{t_1\wedge\tau_\eps\wedge S_M}F(\gamma[t_1\wedge\tau_\eps\wedge S_M])].\]
By definition, we have that
\[R_t=\int^{\infty}_{u_1}  N_t(v_1, v_2)\ud v_1\]
Thus, taking integration, we have
\begin{equation}\label{eqn::integration_martingale}
\E[R_{t_2\wedge\tau_\eps\wedge S_M}F(\gamma[t_1\wedge\tau_\eps\wedge S_M])]=\E[R_{t_1\wedge\tau_\eps\wedge S_M}F(\gamma[t_1\wedge\tau_\eps\wedge S_M])].
\end{equation}
This implies that $\{R_t\}_{t\ge 0}$ is also a local martingale for $\gamma$. It remains to show that $\{R_{t\wedge\tau}\}_{t\ge 0}$ is uniformly integrable and to derive the terminal value~\eqref{eqn::terminalnonintersecting}. 

We denote by $M_t(u)$ the martingale in~\eqref{eqn::mart1}.
We decompose $R_t$ in the following way: for $t<\tau$, we have 
\begin{align}\label{eqn::decR}
R_t=&M_t(v_2)
\times\left(\underbrace{\frac{(g_t(y_1)-g_t(b))(g_t(y_2)-g_t(b))}{(g_t(v_2)-g_t(b))^2}}_{Z_1(t)}\right)^{1-4/\kappa}\notag\\
&\times\underbrace{\int^{\infty}_{g_t(u_1)}\ud v_1 (v_1-g_t(b))^{8/\kappa-2}(v_1-W_t)^{-4/\kappa}(g_t(b)-W_t)^{1-4/\kappa}\times \left(\frac{(v_1-g_t(v_2))^2}{(v_1-g_t(y_1))(v_1-g_t(y_2))}\right)^{4/\kappa}}_{Z_2(t)}.
\end{align}
Let us estimate the terms $Z_1(t)$ and $Z_2(t)$
 one by one. 
\begin{itemize}
\item 
For $Z_1(t)$, we have
\begin{equation}\label{eqn::estimate23}
Z_1(t)\le 1.
\end{equation}
\item 
For $Z_2(t)$, by change of variables $v= \frac{v_1-W_t}{g_t(b)-W_t}$, we have
\begin{align}
Z_2(t)&=\int^{\infty}_{\frac{g_t(u_1)-W_t}{g_t(b)-W_t}} (v-1)^{8/\kappa-2}v^{-4/\kappa}\times \left(\frac{\left(v-\frac{g_t(v_2)-W_t}{g_t(b)-W_t}\right)^2}{\left(v-\frac{g_t(y_1)-W_t}{g_t(b)-W_t}\right)\left(v-\frac{g_t(y_2)-W_t}{g_t(b)-W_t}\right)}\right)^{4/\kappa}\ud v\label{eqn::equality}\\
&\le\int^{\infty}_{\frac{g_t(u_1)-W_t}{g_t(b)-W_t}} (v-1)^{8/\kappa-2}v^{-4/\kappa}\ud v\label{eqn::estimate4}\\
&\le \int_1^{\infty} (v-1)^{8/\kappa-2}v^{-4/\kappa}\ud v=C_{\kappa}.\label{eqn::estimate5}
\end{align}
\end{itemize}
Plugging~\eqref{eqn::estimate23} and~\eqref{eqn::estimate5} into~\eqref{eqn::decR}, we obtain 
\begin{equation}\label{eqn::aux2nonintersecting}
R_t\le C_{\kappa}M_t(v_2).
\end{equation}
This implies that $R_t$ is a uniformly integrable martingale, because $M_t(v_2)$ is a uniformly integrable martingale from  Lemma~\ref{lem::mart1}. 

It remains to derive the terminal value~\eqref{eqn::terminalnonintersecting}. To this end, we analyze the limit of $R_t$ according to the location of $\gamma(\tau)$. There are two cases.
\begin{itemize}
\item Case 1. $\gamma(\tau)\in (b,u_1)$.
In this case, 
by continuity, we have
\[\lim_{t\to\tau} (g_t(u_1)-W_t)=g_\tau(u_1)-W_\tau\quad\text{and}\quad \lim_{t\to\tau} (g_t(b)-W_t)=0.\]
Thus, we have \[\lim_{t\to\tau}\frac{g_t(u_1)-W_t}{g_t(b)-W_t}=\infty.\] Plugging it into
~\eqref{eqn::estimate4}, we have
\[\lim_{t\to\tau}Z_2(t)=0,\quad \text{and}\quad\lim_{t\to\tau}R_t=0.\]
\item Case 2. $\gamma(\tau)\in (u_1, \infty)$. This is the case corresponding to the event $E$.
For $Z_1(t)$, by the same argument as the proof of~\eqref{eqn::mart_decomposition_aux}, we have 
\begin{equation}\label{eqn::Z_2(t)}
\lim_{t\to\tau}Z_1(t)=\frac{(\xi_\gamma(y_1)-\xi_\gamma(b))(\xi_\gamma(y_2)-\xi_\gamma(b))}{(\xi_\gamma(v_2)-\xi_\gamma(b))^2}.
\end{equation}
Note that the right side of~\eqref{eqn::Z_2(t)} is independent of the choice of $\xi_{\gamma}$ if $\xi_{\gamma}(\gamma(\tau))=\infty$.

For $Z_2(t)$, using the expression in~\eqref{eqn::equality} and dominated convergence theorem, we have
\begin{align*}
\lim_{t\to\tau}Z_2(t)=\int_1^{\infty}\ud v(v-1)^{8/\kappa-2}v^{-4/\kappa}=C_{\kappa}. 
\end{align*}
Combining these three observations, we have 
\begin{align*}
\lim_{t\to\tau}R_t&=M_\tau(v_2)\times\left(\frac{(\xi_\gamma(y_1)-\xi_\gamma(b))(\xi_\gamma(y_2)-\xi_\gamma(b))}{(\xi_\gamma(v_2)-\xi_\gamma(b))^2}\right)^{1-4/\kappa}\times C_{\kappa}=R_\tau. 
\end{align*}
\end{itemize}
Combining these two cases, we obtain~\eqref{eqn::terminalnonintersecting} as desired. 
\end{proof}

The following lemma is well-known, for instance, see~\cite[Proposition 5.9]{berestycki2014lectures}. 
For completeness, we provide a proof below.
\begin{lemma}\label{lem::aux2}
Fix $\kappa\in(4,8)$ and fix $b<y_1<y_2<u_2$ and denote by $\gamma\sim\PP^{y_1\to y_2}$ the $\SLE_\kappa$ from $y_1$ to $y_2$. Then, we have
\begin{align}\label{eqn::SLEBU}
&\PP^{y_1\to y_2}[\gamma\text{ hits }(y_2,u_2)\text{ before }(-\infty,b)\cup (u_2,+\infty)]\notag\\
=&\frac{1}{ C_{\kappa}}(y_1-b)^{1-4/\kappa}(y_2-b)^{1-4/\kappa}(y_2-y_1)^{6/\kappa-1}\int^{u_2}_{y_2} f(y_1,y_2;v) (v-b)^{8/\kappa-2}\ud v.
\end{align}
\end{lemma} 
\begin{proof}
We denote by $W$ the driving function of $\gamma$ and by $(g_t, t\ge 0)$ the corresponding conformal maps. Note that by~\cite[Theorem 3]{SchrammWilsonSLECoordinatechanges}, before hitting $(y_2,\infty)$, the law of $\gamma$ equals the law of $\SLE_\kappa(\kappa-6)$  from $y_1$ to $\infty$ with the marked point $y_2$. Define $\tau$ to be the hitting time of $\gamma$ at $(-\infty,b)\cup(y_2,\infty)$. For $v\in (y_2,\infty)$, we define, for $t<\tau$, 
\[N_t(v):=(W_t-g_t(b))^{1-4/\kappa}(g_t(y_2)-g_t(b))^{1-4/\kappa}(g_t(y_2)-W_t)^{6/\kappa-1}g'_t(v) f(W_t,g_t(y_2);g_t(v)) (g_t(v)-g_t(b))^{8/\kappa-2}.\]
From~\cite[Theorem 6]{SchrammWilsonSLECoordinatechanges}, $N_t(v)$ is a local martingale for $\gamma$. 
Moreover, before hitting $[b,+\infty)$, the law of $\gamma$ weighted by $\{N_t(v)\}_{t\ge 0}$ equals $\SLE_\kappa(\kappa-4; 2, -4)$ from $y_0$ to $\infty$ with force points $b$ on its left and $(y_2,v)$ on its right. 
Define
\[L_t:=\frac{1}{ C_{\kappa}}\int_{y_2}^{u_2} N_t(v)\ud v.\]
Then, similarly to~\eqref{eqn::integration_martingale}, we have that $L_t$ is also a local martingale for $\gamma$. Next, we will prove that $\{L_{t\wedge\tau}\}_{t\ge 0}$ is uniformly integrable and derive the terminal value $L_\tau$.

By a change of variables $u=\frac{g_t(v)-g_t(y_2)}{g_t(v)-W_t}$, for $t<\tau$, we have
\begin{align}\label{eqn::uniformbound}
L_t&=\left(\frac{W_t-g_t(b)}{g_t(y_2)-g_t(b)}\right)^{1-4/\kappa}\times\frac{1}{ C_{\kappa}}\int_{0}^{\frac{g_t(u_2)-g_t(y_2)}{g_t(u_2)-W_t}} \left(1-\frac{W_t-g_t(b)}{g_t(y_2)-g_t(b)}u\right)^{8/\kappa-2}u^{-4/\kappa}\ud u\\
&\le\frac{1}{ C_{\kappa}}\int_{0}^{1} \left(1-u\right)^{8/\kappa-2}u^{-4/\kappa}\ud u=1.\notag
\end{align}
Thus, $\{L_{t\wedge\tau}\}_{t\ge 0}$ is uniformly integrable.
 
Now, we derive the terminal value of $L_\tau$. To this end, we analyze the limit of $L_t$ according to the location of $\gamma(\tau)$. There are three cases. 
\begin{itemize}
\item Case 1. $\gamma$ hits $(y_2,u_2)$ before $(-\infty,b)\cup (u_2, \infty)$. In this case, from Lemma~\ref{lem::technical}, we have
\[\lim_{t\to\tau}\frac{W_t-g_t(b)}{g_t(y_2)-g_t(b)}=1, \quad \lim_{t\to\tau}\frac{g_t(u_2)-g_t(y_2)}{g_t(u_2)-W_t}=1.\]
Plugging it into~\eqref{eqn::uniformbound} and using dominated convergence theorem, we have 
\begin{align*}
\lim_{t\to\tau}L_t&=\frac{1}{C_{\kappa}}\int_{0}^{1}(1-v)^{8/\kappa-2}v^{-4/\kappa}\ud v=1. 
\end{align*}
\item Case 2. $\gamma$ hits $(u_2,\infty)$ before $(-\infty, b)$. In this case, from Lemma~\ref{lem::technical}, we have
\[\lim_{t\to\tau}\frac{W_t-g_t(b)}{g_t(y_2)-g_t(b)}=1, \quad \lim_{t\to\tau}\frac{g_t(u_2)-g_t(y_2)}{g_t(u_2)-W_t}=0.\]
Plugging it into~\eqref{eqn::uniformbound}, we have
\[\lim_{t\to\tau}L_t=0.\]
\item Case 3. $\gamma$ hits $(-\infty, b)$ before $(y_2, \infty)$. In this case, from Lemma~\ref{lem::technical}, we have
\[\lim_{t\to\tau}\frac{W_t-g_t(b)}{g_t(y_2)-g_t(b)}=0.\]
Plugging it into~\eqref{eqn::uniformbound}, we have
\[\lim_{t\to\tau}L_t=0.\]
\end{itemize}
Combining these three cases, we have
\[L_\tau=\one_{\{\gamma\text{ hits }(y_2,u_2)\text{ before }(-\infty,b)\cup (u_2,+\infty)\}}.\]
As $\{L_{t\wedge\tau}\}_{t\ge 0}$ is uniformly integrable, we have 
\[\PP^{y_1\to y_2}[\gamma\text{ hits }(y_2,u_2)\text{ before }(-\infty,b)\cup (u_2,+\infty)]=\E^{y_1\to y_2}[L_{\tau}]=L_0,\]
as desired. 
\end{proof}
\begin{proof}[Proof of Proposition~\ref{prop::nonintersecting}]
We use the same notations as in Lemma~\ref{lem::aux1}. 
Denote by $E(\gamma_1)$ the event that $\gamma_1$ first hits $(b,\infty)$ at $(u_1,\infty)$.
Note that 
\begin{align}\label{eqn::aux6nonintersecting}
&\QQ_2\left[\gamma_1\cap\gamma_2=\emptyset, B_1\in (y_0,b), B_2\in (b,y_1), U_1\in (u_1,\infty), U_2\in (y_2,u_2)\right]\notag\\
=&\QQ_2\left[\one_{E(\gamma_1)}\QQ_2\left[\gamma_2\cap\gamma_1=\emptyset, B_2\in (b,y_1), U_2\in (y_2,u_2)\cond\gamma_1\right]\right].
\end{align}
By Lemma~\ref{lem::eta21underQ1}, given $\gamma_1$, the curve $\gamma_2$ is the $\SLE_\kappa$ curve from $y_1$ to $y_2$. We denote by $\xi_{\gamma_1}$ any conformal map from the connected component of $\HH\setminus\gamma_1$ having $y_1$ on the boundary onto $\HH$ such that $\xi_{\gamma_1}(U_1)=\infty$ and we denote by $\gamma\sim\PP^{\xi_{\gamma_1}(y_1)\to\xi_{\gamma_1}(y_2)}$ the $\SLE_\kappa$ curve from $\xi_{\gamma_1}(y_1)$ to $\xi_{\gamma_1}(y_2)$. Then we have 
\begin{align}
&\QQ_2\left[\gamma_2\cap\gamma_1=\emptyset, B_2\in (b,y_1), U_2\in (y_2,u_2)\cond\gamma_1\right]\notag\\
=&\PP^{\xi_{\gamma_1}(y_1)\to\xi_{\gamma_1}(y_2)}[\gamma\text{ hits }(\xi_{\gamma_1}(y_2),\xi_{\gamma_1}(u_2))\text{ before }(-\infty,\xi_{\gamma_1}(b))\cup(\xi_{\gamma_1}(u_2),+\infty)]\notag\\
=&\frac{1}{C_{\kappa}}\frac{(\xi_{\gamma_1}(y_1)-\xi_{\gamma_1}(b))^{1-4/\kappa}(\xi_{\gamma_1}(y_2)-\xi_{\gamma_1}(b))^{1-4/\kappa}}{(\xi_{\gamma_1}(y_2)-\xi_{\gamma_1}(y_1))^{1-6/\kappa}}\notag\\
&\times\int_{\xi_{\gamma_1}(y_2)}^{\xi_{\gamma_1}(u_2)}\ud v f(\xi_{\gamma_1}(y_1), \xi_{\gamma_1}(y_2); v)(v-\xi_{\gamma_1}(b))^{8/\kappa-2}\tag{due to Lemma~\ref{lem::aux2}}\\
=&\frac{1}{C_{\kappa}}\frac{(\xi_{\gamma_1}(y_1)-\xi_{\gamma_1}(b))^{1-4/\kappa}(\xi_{\gamma_1}(y_2)-\xi_{\gamma_1}(b))^{1-4/\kappa}}{(\xi_{\gamma_1}(y_2)-\xi_{\gamma_1}(y_1))^{1-6/\kappa}}\notag\\
&\times\int^{u_2}_{y_2}\xi'_{\gamma_1}(v_2)f(\xi_{\gamma_1}(y_1),\xi_{\gamma_1}(y_2);\xi_{\gamma_1}(v_2))(\xi_{\gamma_1}(v_2)-\xi_{\gamma_1}(b))^{8/\kappa-2}\ud v_2\notag\\
=&\frac{1}{C^2_\kappa}\frac{\xi'_{\gamma_1}(y_1)^{-h}\xi'_{\gamma_1}(y_2)^{-h}}{(\xi_{\gamma_1}(y_2)-\xi_{\gamma_1}(y_1))^{1-6/\kappa}}\times \int^{u_2}_{y_2} R_\tau(v_2)\ud v_2. \tag{\text{due to~\eqref{eqn::terminalnonintersecting}}}
\end{align}
Plugging it into~\eqref{eqn::aux6nonintersecting}, we have 
\begin{align*}
&\QQ_2\left[\gamma_1\cap\gamma_2=\emptyset, B_1\in (y_0,b), B_2\in (b,y_1), U_1\in (u_1,\infty), U_2\in (y_2,u_2)\right]\notag\\
=&\frac{1}{C^2_\kappa}\QQ_2\left[\one_{E(\gamma_1)}\frac{\xi'_{\gamma_1}(y_1)^{-h}\xi'_{\gamma_1}(y_2)^{-h}}{(\xi_{\gamma_1}(y_2)-\xi_{\gamma_1}(y_1))^{1-6/\kappa}}\times \int^{u_2}_{y_2} R_\tau(v_2)\ud v_2\right].
\end{align*}
Denote by $\PP^{y_0}$ the law of the $\gamma\sim\SLE_\kappa$ from $y_0$ to $\infty$ and denote by $E(\gamma)$ the event that $\gamma$ first hits $(b,\infty)$ at $(u_1,\infty)$. 
From Lemma~\ref{lem::eta1underQ1}, $\PP^{y_0}$ weighted by the Poisson kernel to the power $h$ is the same as $\gamma_1$ under $\QQ_2$. Thus,
\begin{align*}
&\QQ_2\left[\gamma_1\cap\gamma_2=\emptyset, B_1\in (y_0,b), B_2\in (b,y_1), U_1\in (u_1,\infty), U_2\in (y_2,u_2)\right]\notag\\
=&\frac{1}{C_{\kappa}G}\E^{y_0}\left[\one_{E(\gamma)}\int^{u_2}_{y_2}R_\tau(v_2)\ud v_2\right]\notag\\
=&\frac{1}{C_{\kappa}G}\int^{u_2}_{y_2}\E^{y_0}\left[\one_{E(\gamma)}R_\tau(v_2)\right]\ud v_2\notag\\
=&\frac{1}{C_{\kappa}G}(b-y_0)^{1-4/\kappa}(y_1-b)^{1-4/\kappa}(y_2-b)^{1-4/\kappa}\notag\\
&\times\int^{u_2}_{y_2}\ud v_2\int^{\infty}_{u_1}\ud v_1 f(y_0,y_1,y_2;v_1,v_2) (v_1-b)^{8/\kappa-2}(v_2-b)^{8/\kappa-2}. \tag{due to Lemma~\ref{lem::aux1}}
\end{align*}
This completes the proof.
\end{proof}

The following asymptotics is immediate from Proposition~\ref{prop::nonintersecting} and we will use it for the proof of Proposition~\ref{prop::global2_mono}. 
\begin{corollary}\label{coro::asy}
Suppose $y_0=0$, $y_1=s$, $y_2=1$ and $b\in (0,1)$. Then, there exists a constant $C=C(\kappa)$, such that for every $b\in (0,1)$, 
\begin{equation}\label{eqn::probestimate}
\lim_{s\to 0}\frac{1}{s^{2-8/\kappa}}\QQ_2[\gamma_1\cap\gamma_2=\emptyset, B_1\in (0,bs), B_2\in (bs,s)]=C\times b^{1-4/\kappa}(1-b)^{1-4/\kappa}.
\end{equation}
\end{corollary}
\begin{proof}
Note that by Proposition~\ref{prop::nonintersecting}, we have
\begin{align}
&\QQ_2[\gamma_1\cap\gamma_2=\emptyset, B_1\in (0,bs), B_2\in (bs,s)]\notag\\
=&\frac{1}{C_{\kappa}\int^{\infty}_{1} \ud v f(0,s,1;v)}\label{eqn::restrictionversion}\\
&\times\underbrace{\int^{\infty}_{1}\ud v_2\int^{\infty}_{v_2}\ud v_1 f(0,s,1;v_1,v_2)\times (bs)^{1-4/\kappa}(s-bs)^{1-4/\kappa}(1-bs)^{1-4/\kappa}(v_1-bs)^{8/\kappa-2}(v_2-bs)^{8/\kappa-2}}_{Z(s)}. \notag
\end{align}
By dominated convergence theorem, we have 
\begin{align*}
&\lim_{s\to 0}\frac{1}{s^{2/\kappa}}\int^{\infty}_{1} \ud v f(0,s,1;v)=\int_1^{\infty}\ud v(v-1)^{-4/\kappa}v^{-8/\kappa}=B(12/\kappa-1, 1-4/\kappa).\\
&\lim_{s\to 0}\frac{1}{s^{2-6/\kappa}}Z(s)
=b^{1-4/\kappa}(1-b)^{1-4/\kappa}\int^{\infty}_{1}\ud v_2\int^{\infty}_{v_2}\ud v_1 (v_1-1)^{-4/\kappa}(v_2-1)^{-4/\kappa}(v_2-v_1)^{8/\kappa}v_1^{-2}v_2^{-2}.
\end{align*}
Plugging these two asymptotics into~\eqref{eqn::restrictionversion}, we obtain~\eqref{eqn::probestimate}.
\end{proof}

\subsection{Proof of Proposition~\ref{prop::global2_mono}}
\label{subsec::global2_mono}
We will prove Proposition~\ref{prop::global2_mono} using Corollary~\ref{coro::asy}. 
From the conformal invariance, we may assume that $\Omega=\HH$ and $x_1=\infty, x_2=0, x_4=1$. With such normalization, $s=x_3$ is the cross-ratio (see Lemma~\ref{lem::Ls}).  We sample $(\gamma_1,\gamma_2)$ as in Definition~\ref{def::Q2}. We need to estimate the probability
\[\QQ_2[\gamma_1\cap\gamma_2=\emptyset].\]
Recall from Proposition~\ref{prop::nonintersecting} that 
$\tau_1$ (resp. $\tau_2$) is the fist hitting time of $\gamma_1$ (resp. $\gamma_2$) at $(1,\infty)$ and that 
$B_1$ (resp. $B_2$) is the last hitting point of $\gamma_1$ (resp. $\gamma_2$) at $(0,s)$ before $\tau_1$ (resp. $\tau_2$). From Corollary~\ref{coro::asy}, we already have the asymptotics of the probability: 
\[\QQ_2[\gamma_1\cap\gamma_2=\emptyset, B_1\in (0,bs), B_2\in (bs,s)].\]
We will show in the following lemma that these two probabilities are comparable as $s\to 0$. 

\begin{lemma}\label{lem::fullprobability}
For every $b\in (0,1)$, there exists a constant $C=C(b)\in(0,1)$ such that, for all $s>0$ small,
\begin{equation}\label{eqn::full3}
\QQ_2[\gamma_1\cap\gamma_2=\emptyset, B_1\in (0,bs), B_2\in (bs,s)]\ge C\QQ_2[\gamma_1\cap\gamma_2=\emptyset].
\end{equation}
\end{lemma}
\begin{proof}
We claim that there exists a constant $C$ such that 
\begin{align}\label{eqn::full1}
\QQ_2[\gamma_1\cap\gamma_2=\emptyset, B_1\in (0,bs), B_2\in (bs,s)]\ge C\QQ_2[\gamma_1\cap\gamma_2=\emptyset, B_1\in (0,bs)], \\
\label{eqn::full2}
\QQ_2[\gamma_1\cap\gamma_2=\emptyset, B_1\in (0,bs), B_2\in (bs,s)]\ge C\QQ_2[\gamma_1\cap\gamma_2=\emptyset, B_2\in (bs,s)].
\end{align}
We first prove~\eqref{eqn::full1}. 
We choose $\xi_{\gamma_1}$ as the same conformal map in the proof of Proposition~\ref{prop::nonintersecting}.
We denote by $\gamma\sim\PP^{\xi_{\gamma_1}(s)\to\xi_{\gamma_1}(1)}$ the $\SLE_\kappa$ from $\xi_{\gamma_1}(s)$ to $\xi_{\gamma_1}(1)$. 
By Lemma~\ref{lem::aux2}, on the event $\{B_1\in(0,bs)\}$, we have
\begin{align*}
&\QQ_2[\gamma_2\cap\gamma_1=\emptyset,B_2\in (bs,s)\cond\gamma_1]\\
=&\PP^{\xi_{\gamma_1}(s)\to\xi_{\gamma_1}(1)}[\gamma\text{ hits }(\xi_{\gamma_1}(1),+\infty)\text{ before }(-\infty,\xi_{\gamma_1}(bs))]\\
=&\frac{1}{ C_{\kappa}}(\xi_{\gamma_1}(s)-\xi_{\gamma_1}(bs))^{1-4/\kappa}(\xi_{\gamma_1}(1)-\xi_{\gamma_1}(bs))^{1-4/\kappa}(\xi_{\gamma_1}(1)-\xi_{\gamma_1}(s))^{6/\kappa-1}\\&\times\int^{\infty}_{\xi_{\gamma_1}(1)}\ud v f(\xi_{\gamma_1}(s),\xi_{\gamma_1}(1);v) (v-\xi_{\gamma_1}(bs))^{8/\kappa-2},
\end{align*}
and 
\begin{align*}
&\QQ_2[\gamma_2\cap\gamma_1=\emptyset\cond\gamma_1]\\
=&\PP^{\xi_{\gamma_1}(s)\to\xi_{\gamma_1}(1)}[\gamma\text{ hits }(\xi_{\gamma_1}(1),+\infty)\text{ before }(-\infty,\xi_{\gamma_1}(B_1))]\\
=&\frac{1}{ C_{\kappa}}(\xi_{\gamma_1}(s)-\xi_{\gamma_1}(B_1))^{1-4/\kappa}(\xi_{\gamma_1}(1)-\xi_{\gamma_1}(B_1))^{1-4/\kappa}(\xi_{\gamma_1}(1)-\xi_{\gamma_1}(s))^{6/\kappa-1}\\&\times\int^{\infty}_{\xi_{\gamma_1}(1)}\ud v f(\xi_{\gamma_1}(s),\xi_{\gamma_1}(1);v) (v-\xi_{\gamma_1}(B_1))^{8/\kappa-2}.
\end{align*}
Thus, on the event $\{B_1\in(0,bs)\}$, we have
\[\frac{\QQ_2[\gamma_2\cap\gamma_1=\emptyset,B_2\in (bs,s)\cond\gamma_1]}{\QQ_2[\gamma_2\cap\gamma_1=\emptyset\cond\gamma_1]}\ge\left(\frac{\xi_{\gamma_1}(s)-\xi_{\gamma_1}(bs)}{\xi_{\gamma_1}(s)-\xi_{\gamma_1}(B_1)}\right)^{1-4/\kappa}\left(\frac{\xi_{\gamma_1}(1)-\xi_{\gamma_1}(bs)}{\xi_{\gamma_1}(1)-\xi_{\gamma_1}(B_1)}\right)^{1-4/\kappa}.\]
It suffices to prove that, on the event $\{B_1\in(0,bs)\}$,  there exists a constant $C=C(b)$ such that 
\begin{equation}\label{eqn::comparison}
\frac{\xi_{\gamma_1}(s)-\xi_{\gamma_1}(bs)}{\xi_{\gamma_1}(s)-\xi_{\gamma_1}(B_1)}\times\frac{\xi_{\gamma_1}(1)-\xi_{\gamma_1}(bs)}{\xi_{\gamma_1}(1)-\xi_{\gamma_1}(B_1)}\ge C.
\end{equation}

We denote by $W$ the driving function of $\gamma_1$ and by $(g_t, t\ge 0)$ the coresponding conformal maps. We choose a point $b'\in(0,bs)$  and define 
\[F_t:=\frac{g_t(s)-g_t(bs)}{g_t(s)-g_t(b')}.\]
Note that 
\[\partial_t\log F_t=\frac{2(g_t(bs)-g_t(b'))}{(g_t(s)-W_t)(g_t(bs)-W_t)(g_t(b')-W_t)}\ge 0.\]
Thus, we have $F_t$ is increasing. In particular, $F_{\tau_1}:=\lim_{t\to\tau_1}F_t\ge F_0$. This implies
\[\frac{\xi_{\gamma_1}(s)-\xi_{\gamma_1}(bs)}{\xi_{\gamma_1}(s)-\xi_{\gamma_1}(b')}\ge 1-b.\]
By letting $b'\to B_1$, we have 
\begin{equation}\label{eqn::comp1}
\frac{\xi_{\gamma_1}(s)-\xi_{\gamma_1}(bs)}{\xi_{\gamma_1}(s)-\xi_{\gamma_1}(B_1)}\ge 1-b.
\end{equation}
Similarly, we have
\begin{equation}\label{eqn::comp2}
\frac{\xi_{\gamma_1}(1)-\xi_{\gamma_1}(bs)}{\xi_{\gamma_1}(1)-\xi_{\gamma_1}(B_1)}\ge 1-bs.
\end{equation}
By~\eqref{eqn::comp1} and~\eqref{eqn::comp2}, we derive~\eqref{eqn::comparison}.
This completes the proof of~\eqref{eqn::full1}. Note that~\eqref{eqn::full2} can be proved similarly. 

For~\eqref{eqn::full3}, we have 
\begin{align*}
\QQ_2[\gamma_1\cap\gamma_2=\emptyset]&=\QQ_2[\gamma_1\cap\gamma_2=\emptyset,B_1\in (0,bs)]+\QQ_2[\gamma_1\cap\gamma_2=\emptyset, B_1\in (bs,s)]\\
&\le \QQ_2[\gamma_1\cap\gamma_2=\emptyset,B_1\in (0,bs)]+\QQ_2[\gamma_1\cap\gamma_2=\emptyset, B_2\in (bs,s)].
\end{align*}
Combining with~\eqref{eqn::full1} and~\eqref{eqn::full2}, we complete the proof.
\end{proof}
\begin{proof}[Proof of Proposition~\ref{prop::global2_mono}]
This is a combination of Corollary~\ref{coro::asy}, Lemma~\ref{lem::fullprobability} and Lemma~\ref{lem::Ls}. 
\end{proof}

\section{Proof of Proposition~\ref{prop::FKIsing_mono}}
\label{sec::FKIsing_mono}
\subsection{Preliminaries}
\subsubsection*{Random-cluster model}
We focus on the square lattice $\Z^2$, which is the graph with vertex set $V(\Z^2):=\{ z = (m, n) \colon m, n\in \Z\}$ and edge set $E(\Z^2)$ given by edges between 
those vertices whose Euclidean distance equals one (called neighbors). 
This is our primal lattice. Its standard dual lattice is denoted by $(\Z^2)^{\bullet}$. 
The medial lattice $(\Z^2)^{\diamond}$ is the graph with centers of edges of $\Z^2$ as its vertex set and edges connecting 
neighbors. 
For a subgraph $\graph \subset \Z^2$ (resp.~of $(\Z^2)^{\bullet}$ or $(\Z^2)^{\diamond}$), we define its boundary to be the following set of vertices:
\begin{align*}
	\partial \graph = \{ z \in V(\graph) \, \colon \, \exists \; w \not\in V(\graph) \text{ such that }\edge{z}{w}\in E(\Z^2) \text{ (resp., $\edge{z}{w}\in E((\Z^2)^{\bullet})$ or $\edge{z}{w}\in E((\Z^2)^{\diamond})$ )}\} .
\end{align*}
When we add the subscript or superscript $\delta$, we mean that subgraphs of the lattices $\Z^2, (\Z^2)^{\bullet}, (\Z^2)^\diamond$ have been scaled by $\delta > 0$. 

Let $\graph = (V(\graph), E(\graph))$ be a finite subgraph of $\Z^2$.
A random-cluster {configuration} 
$\omega=(\omega_e)_{e \in E(\graph)}$ is an element of $\{0,1\}^{E(\graph)}$.
An edge $e \in E(\graph)$ is said to be {open} (resp.~{closed}) if $\omega_e=1$ (resp.~$\omega_e=0$).
We view the configuration $\omega$ 
as a subgraph of $\graph$ with vertex set $V(\graph)$  and edge set $\{e\in E(\graph) \colon \omega_e=1\}$.
We denote by $o(\omega)$ (resp.~$c(\omega)$) the number of open (resp.~closed) edges in~$\omega$.
The boundary conditions encode how the vertices are connected outside of $\graph$.
Precisely, by a {boundary condition} $\bssymb$ we refer to a partition $\bssymb_1 \sqcup \cdots \sqcup \bssymb_m$ of $\partial \graph$.
Two vertices $z,w \in \partial \graph$ are said to be {wired} in $\bssymb$ if $z,w \in \bssymb_j$ for some $j$. 
In contrast, {free} boundary segments comprise vertices that are not wired with any other vertex (so the corresponding part $\pi_j$ is a singleton).  
We denote by $\omega^{\bssymb}$
the (quotient) graph obtained from the configuration $\omega$ by identifying the wired vertices in $\bssymb$.
The {random-cluster model} on $\graph$ with edge-weight $p\in [0,1]$, cluster-weight $q>0$,
and boundary condition $\bssymb$, is the probability measure $\smash{\PRCM^{\bssymb}_{p,q,\graph}}$ on
the set $\{0,1\}^{E(\graph)}$ of configurations $\omega$  defined by
\begin{align*}
	\PRCM^{\bssymb}_{p,q,\graph}[\omega] 
	:= \; & \frac{p^{o(\omega)}(1-p)^{c(\omega)}q^{k(\omega^{\bssymb})}}{\underset{\omega \in \{0,1\}^{E(\graph)}}{\sum} p^{o(\omega)}(1-p)^{c(\omega)}q^{k(\omega^{\bssymb})} } ,
\end{align*}
where $k(\omega^{\bssymb})$ is the number of 
connected components of the graph $\omega^{\bssymb}$.
It has been proven for the range $q \in [1,4]$~\cite{DuminilSidoraviciusTassionContinuityPhaseTransition} 
that the critical edge-weight is given by
\begin{align*}
	p = p_c(q) := \frac{\sqrt{q}}{1+\sqrt{q}}, 
\end{align*}
in the sense that for $p > p_c(q)$ there almost surely exists an infinite cluster, while for $p < p_c(q)$ there does not.
In this section, we focus on the critical FK-Ising model, that is, the random-cluster model with $q=2$ and $p=p_c(2)$.
 
A {discrete quad} is a simply connected subgraph $\Omega$ of $\Z^2$, 
 {or $\delta \Z^2$,}
 with four marked boundary points $x_1, x_2, x_3 , x_{4}$ in counterclockwise order, whose precise definition is given below. Firstly, we define the {medial quad}. We give orientation to edges of the medial graph $(\Z^2)^\diamond$ as follows: edges of each face containing a primal vertex have clockwise orientation, while edges of each face containing a dual vertex have counterclockwise orientation. 
 Let $x_1^\diamond,x_2^{\diamond},x_3^{\diamond}, x_{4}^\diamond$ be four distinct medial vertices. Let $(x_2^\diamond \, x_1^\diamond), (x_2^\diamond \, x_3^\diamond), (x_4^\diamond \, x_3^\diamond) , (x_{4}^\diamond  \, x_{1}^\diamond)$ be four oriented paths on $(\Z^2)^\diamond$ satisfying the following four conditions (we use the convention that $x_5^{\diamond}:=x_1^{\diamond}$): (1) all of the paths are edge-self-avoiding and satisfy $(x_{2j}^{\diamond}x_{2j-1}^{\diamond})\cap (x_{2j}^{\diamond}x_{2j+1}^{\diamond})=\{x_{2j}^{\diamond}\}$, $(x_{2j}^{\diamond}x_{2j+1}^{\diamond})\cap (x_{2j+2}^{\diamond}x_{2j+1}^{\diamond})=\{x_{2j+1}^{\diamond}\}$, and any other pairs in these four paths do not have common vertex;
(2) the infinite connected component of $(\Z^2)^\diamond\setminus \cup_{1\leq j\leq 2}\big(
	(x_{2j}^\diamond \, x_{2j+1}^\diamond)\cup (x_{2j}^{\diamond}x_{2j-1}^{\diamond})\big)$ lies on the left (resp., right) of the oriented paths $(x_2^\diamond \, x_1^\diamond)$ and $(x_4^{\diamond} \, x_3^{\diamond})$ (resp., $(x_2^{\diamond}\, x_3^{\diamond})$ and $(x_4^{\diamond}\, x_1^{\diamond})$).  Given $(x_2^\diamond \, x_1^\diamond), (x_2^\diamond \, x_3^\diamond), (x_4^\diamond \, x_3^\diamond) , (x_{4}^\diamond  \, x_{1}^\diamond)$, the medial quad $(\Omega^\diamond; x_1^\diamond,x_2^{\diamond},x_3^{\diamond}, x_{4}^\diamond)$ is defined 
 as the subgraph of $(\Z^2)^\diamond$ induced by the vertices lying on or enclosed by the non-oriented loop obtained by concatenating all of these four paths. 
 Secondly, we define the {primal quad} 
 $(\Omega;x_1,x_2,x_3,x_{4})$ induced by $(\Omega^\diamond;x_1^\diamond,x_2^{\diamond},x_3^{\diamond},x_{4}^\diamond)$ as follows: (1) the edge set $E(\Omega)$ consists of edges passing through endpoints of medial edges in $E(\Omega^\diamond)\setminus \left((x_2^{\diamond}x_1^{\diamond})\cup (x_4^{\diamond}x_3^{\diamond})\right)$, while the vertex set $V(\Omega)$ consists of endpoints of edges in $E(G)$;
(2) the arc $(x_{2j}x_{2j+1})$ is defined to be the set of edges lying outside of $\Omega^{\diamond}$ whose midpoints are vertices in $(x_{2j}^{\diamond}x_{2j+1}^{\diamond})$, and the arc $(x_{2j}x_{2j+1})$ is oriented to have $\Omega^{\diamond}$ on its left, for $1\leq j\leq 2$ (we use the convention that $x_5:=x_1$). 
For $j\in \{1,2,3,4\}$, the outer corner $y_j^{\diamond}\in (\mathbb{Z}^2)^{\diamond}\setminus \Omega^{\diamond}$ is defined to be the medial vertex lying on the face  containing $x_j$ that is adjacent to $x_j^{\diamond}$. See Figure~\ref{fig::quad} for an illustration. 

 \begin{figure}[ht!]
 	\begin{center}
 		\includegraphics[width=0.75\textwidth]{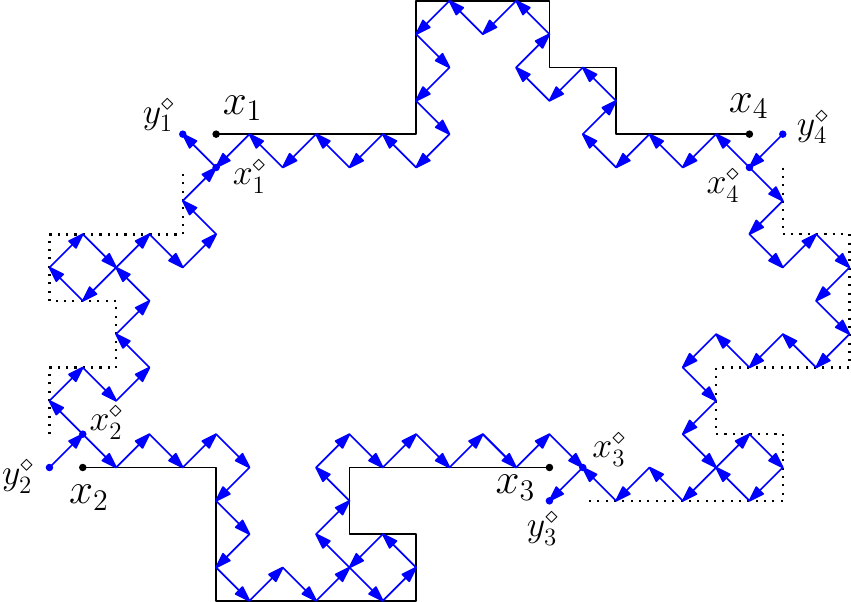}
 	\end{center}
 	\caption{Illustration of a discrete quad. The blue lines with arrows represent oriented medial edges. And the blue loop is the boundary of the medial quad $\Omega^{\diamond}$. 
 		The black lines represent edges of the primal lattice, while the dashed lines represent edges of the dual lattice. }
 		 	\label{fig::quad}
 \end{figure}

 We shall focus on the critical FK-Ising model on the primal quad $(\Omega^\delta; x_1^\delta,x_2^{\delta},x_3^{\delta}, x_{4}^\delta)$ with the alternating boundary conditions~\eqref{eqn::bc}. We denote by $\mathbb{P}^{\delta}$ the corresponding probability measure. Recall that
 \begin{itemize}
 	\item We denote by $\conn^{(1)}=\conn^{(1)}(\Omega^{\delta}; x_1^{\delta}, x_2^{\delta}, x_3^{\delta}, x_4^{\delta})$ the event that there exists an open path in $\Omega^{\delta}$ connecting the boundary arcs $(x_2^{\delta}x_3^{\delta})$ and $(x_4^{\delta}x_1^{\delta})$. 
 	\item We denote by $\conn^{(2)}=\conn^{(2)}(\Omega^{\delta}; x_1^{\delta}, x_2^{\delta}, x_3^{\delta}, x_4^{\delta})$ the event that there exist two disjoint open paths in $\Omega^{\delta}$ connecting the boundary arcs $(x_2^{\delta}x_3^{\delta})$ and $(x_4^{\delta}x_1^{\delta})$. 
 \end{itemize}

 Now let $\omega\in \{0,1\}^{E(\Omega^\delta)}$ be a configuration of the FK-Ising model on $(\Omega^{\delta};x_1^{\delta},x_2^{\delta},x_3^{\delta},x_{4}^{\delta})$ with the boundary condition~\eqref{eqn::bc}. Given $\omega$, we can draw edge-self-avoiding loops on $\Omega^{\delta, \diamond}$ as follows: a loop arriving at a vertex of $\Omega^{\delta,\diamond }$ always makes a turn of $\pm\pi/2$, so as not to cross the open or dual open edges through this vertex. Given $\omega$, the {loop representation} contains loops and $2$ pairwise edge-disjoint and edge-self-avoiding {interfaces} connecting the four outer corners $y_{1}^{\delta,\diamond}, y_2^{\delta,\diamond},y_3^{\delta,\diamond},y_{4}^{\delta,\diamond}$ of  the medial polygon $(\Omega^{\delta,\diamond};x_1^{\delta,\diamond},x_2^{\delta,\diamond},x_3^{\delta,\diamond}, x_{4}^{\delta,\diamond})$. Note that conditional on the event $\conn^{(1)}$, the two interfaces connect the outer corners $y_2^{\delta,\diamond}$ to $y_1^{\delta,\diamond}$ and $y_4^{\delta,\diamond}$ to $y_3^{\delta,\diamond}$, respectively; we denote by $\eta_1^{\delta}$ and $\eta_2^{\delta}$ these two interfaces.

\subsubsection*{Convergence of crossing probability for $\conn^{(1)}$}

Fix a quad $(\Omega;x_1,x_2,x_3,x_4)$. 
We say that a sequence of discrete quads $(\Omega^{\delta}; x_1^{\delta}, x_2^{\delta},x_3^{\delta}, x_{4}^{\delta})$
converges as $\delta \to 0$ to $(\Omega; x_1, x_2,x_3, x_{4})$ in the {Carath\'{e}odory sense}
if there exist conformal maps $\varphi_{\delta}$ from $\Omega^{\delta}$ onto $\U$,
and a conformal map $\varphi$ from $\Omega$ onto $\U$,
such that $\varphi_{\delta}^{-1} \to \varphi^{-1}$ locally uniformly on $\U$,
and $\varphi_{\delta}(x_j^{\delta}) \to \varphi(x_j)$ for all $1\le j\le 4$. 
Under such convergence, we have the convergence of the crossing probability~\cite{ChelkakSmirnovIsing} (see also~\cite[Corollary~2.7]{IzyurovObservableFree} and~\cite[Eq.~(117)]{FloresSimmonsKlebanZiffCrossingProba}): 
\begin{align}\label{eqn::cross_proba_1}
		p^{(1)}_{\mathrm{FKIsing}}(\Omega; x_1, x_2, x_3, x_4):=&
		\lim_{\delta\to 0}\PP^{\delta}\left[\conn^{(1)}(\Omega^{\delta}; x_1^{\delta}, x_2^{\delta}, x_3^{\delta}, x_4^{\delta})\right]=\frac{\sqrt{s}+1-\sqrt{1-s}}{1+\sqrt{s}}.
	\end{align}
Denote by $L$ the extremal distance between $(x_2x_3)$ and $(x_4x_1)$ in $\Omega$, then it follows from~\eqref{eqn::cross_proba_1} and Lemma~\ref{lem::Ls} that
\begin{equation} \label{eqn::asy_FK_1}
	p^{(1)}_{\mathrm{FKIsing}}(L):=p^{(1)}_{\mathrm{FKIsing}}(\Omega; x_1, x_2, x_3, x_4)\sim \exp(-L/2), \quad\text{as }L\to\infty.
\end{equation}

\subsubsection*{Convergence of the pair of interfaces}

Next, we state the convergence of the pair of interfaces $(\eta_1^{\delta},\eta_2^{\delta})$ conditional on the event $\conn^{(1)}$.  
Recall that $X$ denotes the set of planar oriented curves and we equip $X$ with metric~\eqref{eqn::curve_metric}. Now we consider the space of pairs of curves. We endow $X^{2}$ with the following metric
\begin{equation}\label{eq::curve_metric_2} 
	\metric\left(\left(\eta_j\right)_{1\leq j\leq 2}, \left(\tilde{\eta}_j\right)_{1\leq j\leq 2}\right):=\sup_{1\leq j\leq 2}\dist\left(\eta_j,\tilde{\eta}_j\right),
\end{equation}
where $\dist(\eta, \tilde{\eta})$ is defined in~\eqref{eqn::curve_metric}.
Then the metric space $(X^{2},\metric)$ is complete and separable (see e.g.,~\cite[Section~3.3.1]{KarrilaMultipleSLELocalGlobal}).

Note that Carath\'{e}odory convergence allows wild behavior of the boundaries around the marked points. 
In order to ensure
precompactness of the interfaces, we need a convergence on polygons stronger than the above Carath\'{e}odory convergence. The following notion was introduced by Karrila in~\cite{KarrilaConformalImage}. See also~\cite{KarrilaMultipleSLELocalGlobal} and~\cite{ChelkakWanMassiveLERW}.
Fix a quad $(\Omega; x_1, x_2, x_3, x_4)$. We say that a sequence of discrete quads $(\Omega^{\delta}; x_1^{\delta}, x_2^{\delta},x_3^{\delta}, x_{4}^{\delta})$  
	converges as $\delta \to 0$ to $(\Omega; x_1, x_2,x_3, x_{4})$ in the {close-Carath\'{e}odory sense} if it converges in the Carath\'{e}odory sense, and in addition, for each $j\in\{1, 2,3, 4\}$, we have $x_j^{\delta}\to x_j$ as $\delta\to 0$ and the following is fulfilled: 
	Given a reference point $u\in\Omega$ and 
	$r>0$ small enough, let $S_r$ be the arc of $\partial B(x_j,r)\cap\Omega$ disconnecting (in $\Omega$) $x_j$ from $u$ and from all other arcs of this set. We require that, for each $r$ small enough and for all sufficiently small $\delta$ (depending on $r$), the boundary point $x_j^{\delta}$ is connected to the midpoint of $S_r$ inside $\Omega^{\delta}\cap B(x_j,r)$.

\begin{lemma} \label{lem::cvg_global_sle}
Fix a quad $(\Omega;x_1,x_2,x_3,x_4)$ and suppose that $(\Omega^{\delta};x_1^{\delta},x_2^{\delta},x_3^{\delta},x_4^{\delta})$ is a sequence of discrete quads that converges to $(\Omega;x_1,x_2,x_3,x_4)$ in the close-Carath\'{e}odory sense. Then as $\delta\to 0$, conditional on the event $\vartheta^{(1)}(\Omega^{\delta};x_1^{\delta},x_2^{\delta},x_3^{\delta},x_4^{\delta})$, the pair of interfaces $(\eta_1^{\delta},\eta_2^{\delta})$ converges in distribution to the global 2-$\SLE_{16/3}$ in $X_0(\Omega; x_1, x_2, x_3, x_4)$
under the topology induced by~\eqref{eq::curve_metric_2}.
\end{lemma}
\begin{proof}
	This is a special case of~\cite[Proposition~1.4]{BeffaraPeltolaWuUniqueness}.
\end{proof}

\subsection{Proof of Proposition~\ref{prop::FKIsing_mono}}
\begin{proof}[Proof of Proposition~\ref{prop::FKIsing_mono}]
Fix a quad $(\Omega;x_1,x_2,x_3,x_4)$ and suppose $(\Omega^{\delta};x_1^{\delta},x_2^{\delta},x_3^{\delta},x_{4}^{\delta})$ is a sequence of discrete quads that converges to $(\Omega;x_1^{\delta},x_2^{\delta},x_3^{\delta},x_{4}^{\delta})$ in the close-Carath\'{e}odory sense. 
\begin{itemize}
\item Suppose that $(\eta_1,\eta_2)\in X_0(\Omega;x_1,x_2,x_3,x_4)$ is  a global $2$-$\mathrm{SLE}_{16/3}$ such that $\eta_1$ goes from $x_2$ to $x_1$ and $\eta_2$ goes from $x_3$ to $x_4$. We denote by $\mathbb{Q}_2$ the law of $(\eta_1,\eta_2)$. 
\item We denote by $\mathbb{Q}^{\delta}$ the measure $\mathbb{P}^{\delta}$ conditional on the event $\vartheta^{(1)}(\Omega^{\delta};x_1^{\delta},x_2^{\delta},x_3^{\delta},x_4^{\delta})$. 
\end{itemize}
Thanks to Proposition~\ref{prop::global2_mono}, it suffices to show that
	\begin{equation}\label{eqn::FK_mono_aux1}
		\lim_{\delta\to 0} \mathbb{Q}^{\delta}\left[\vartheta^{(2)}(\Omega^{\delta};x_1^{\delta},x_2^{\delta},x_3^{\delta},x_4^{\delta})\right]= \mathbb{Q}_2\left[\eta_1\cap \eta_2=\emptyset\right].
	\end{equation}
	
From the convergence in Lemma~\ref{lem::cvg_global_sle}, we may couple $(\eta_1^{\delta},\eta_2^{\delta})$ under $\mathbb{Q}^{\delta}$ and $(\eta_1,\eta_2)$ under $\QQ_2$ into the same probability space so that $(\eta_1^{\delta},\eta_2^{\delta})$ converges to $(\eta_1,\eta_2)$ under the metric~\eqref{eq::curve_metric_2} almost surely. We denote by $\mathbb{Q}$ the coupled joint law. For two sets $A,B\subseteq \mathbb{C}$, we define 
\begin{equation*}
	\Dist(A,B):=\inf_{z_1\in A, z_2\in B} |z_1-z_2|.
\end{equation*}
Note that for any fixed $r>0$, when $\delta$ is small enough, we have 
	\begin{equation*}
\{\Dist(\eta_1^{\delta},\eta_2^{\delta})\geq r\}\subseteq \vartheta^{(2)}(\Omega^{\delta};x_1^{\delta},x_2^{\delta},x_3^{\delta},x_4^{\delta}).
	\end{equation*}
	From the almost sure convergence of $(\eta_1^{\delta},\eta_2^{\delta})$ to $(\eta_1,\eta_2)$, we have
	\begin{equation*}
\lim_{r\to 0}\lim_{\delta\to 0} \mathbb{Q}\left[\Dist(\eta_1^{\delta},\eta_2^{\delta})\geq r\right] = \lim_{r\to 0} \mathbb{Q}\left[\Dist(\eta_1,\eta_2)\geq r\right]= \mathbb{Q}\left[\eta_1\cap \eta_2=\emptyset\right].
	\end{equation*}
To show~\eqref{eqn::FK_mono_aux1}, it remains to show 
	\begin{equation}\label{eqn::FK_mono_aux2}
		\limsup_{r\to 0}\limsup _{\delta\to 0} \mathbb{Q}^{\delta}\left[\vartheta^{(2)}(\Omega^{\delta};x_1^{\delta},x_2^{\delta},x_3^{\delta},x_4^{\delta})\cap \{\Dist(\eta_1^{\delta},\eta_2^{\delta})<r\}\right]=0.
	\end{equation}
	 
For small enough $\epsilon>0$, we define
	\begin{equation*}
		\Omega_\epsilon:=\left\{z\in \Omega: \inf_{w\in \partial \Omega}|z-w|\geq \epsilon\right\}. 
	\end{equation*}
According to Lemma~\ref{lem::FK_mono_aux1} with $\kappa=16/3$, we have, for $j=1,2$,
\begin{align*}
	\lim_{\epsilon\to 0}\lim_{r\to 0} \mathbb{Q}\left[\Dist(\eta_j\setminus \Omega_{\epsilon},\eta_{3-j})\leq r\right]=\lim_{\epsilon\to 0} \mathbb{Q}\left[(\eta_1\cap \eta_2)\setminus \Omega_\epsilon\neq \emptyset\right]=\mathbb{Q}\left[(\eta_1\cap\eta_2\cap \partial \Omega)\neq \emptyset\right]=0.
\end{align*}From the almost sure convergence of $(
\eta_1^{\delta},\eta_2^{\delta})$ to $(\eta_1,\eta_2)$, we have, for $j=1,2$,
\begin{align*}
\limsup_{\epsilon\to 0}\limsup_{r\to 0}\limsup_{\delta\to 0} \mathbb{Q}^{\delta}\left[\Dist(\eta_j^{\delta}\setminus \Omega_{\epsilon},\eta_{3-j}^{\delta})< r\right]\leq \lim_{\epsilon\to 0}\lim_{r\to 0} \mathbb{Q}\left[\Dist(\eta_j\setminus \Omega_{\epsilon},\eta_{3-j})\leq r\right]=0.
\end{align*}Thus, it suffices to show that, for $\eps>0$ small enough, we have
\begin{equation} \label{eqn::FK_mono_aux3}
	\limsup_{r\to 0}\limsup _{\delta\to 0} \mathbb{Q}^{\delta}\left[\vartheta^{(2)}(\Omega^{\delta};x_1^{\delta},x_2^{\delta},x_3^{\delta},x_4^{\delta})\cap \{\Dist(\eta_1^{\delta}\cap \Omega_{\epsilon},\eta_2^{\delta}\cap\Omega_{\epsilon})<r\}\right]=0.
\end{equation}

We will show in Lemma~\ref{lem::FK_mono_aux2} that there exist two constants $c_1,c_2\in (0,\infty)$ depending only on $(\Omega;x_1,\ldots,x_4)$ such that for all $r\in (0,\epsilon/10)$ and small enough $\delta>0$, we have 
	\begin{equation} \label{eqn::FK_mono_aux4}
	\mathbb{Q}^{\delta}\left[\vartheta^{(2)}(\Omega^{\delta};x_1^{\delta},x_2^{\delta},x_3^{\delta},x_4^{\delta})\cap \{\Dist(\eta_1^{\delta}\cap \Omega_{\epsilon},\eta_2^{\delta}\cap\Omega_{\epsilon})<r\}\right]\leq c_1\times \frac{r^{c_2}}{\epsilon^{2+c_2}}.
	\end{equation}
Assume~\eqref{eqn::FK_mono_aux4} is true. Letting $\delta\to 0$ first and then $r\to 0$, we obtain~\eqref{eqn::FK_mono_aux3}, as desired.
\end{proof}

\begin{lemma}\label{lem::FK_mono_aux2}
There exist two constants $c_1,c_2\in (0,\infty)$ depending only on $(\Omega;x_1,\ldots,x_4)$ such that for all $r\in (0,\epsilon/10)$ and small enough $\delta>0$, we have~\eqref{eqn::FK_mono_aux4}. 
\end{lemma}

\subsection{Proof of Lemma~\ref{lem::FK_mono_aux2}}
The goal of this section is to prove Lemma~\ref{lem::FK_mono_aux2}. In this section, we assume the same setup as in the proof for Proposition~\ref{prop::FKIsing_mono}. We emphasize that the conclusion of Lemma~\ref{lem::FK_mono_aux2} still holds if we consider general critical random-cluster models on $\Omega^{\delta}$ with $q\in[1,4)$, while the constants $c_1, c_2$ will then also depend on $q$. 
The idea is as follows: the event $\{\eta_1^{\delta} \text{ and }\eta_2^{\delta}\text{ are vertex-disjoint}\}\cap \{\Dist\left(\eta_1^{\delta}\cap \Omega_{\epsilon}, \eta_2^{\delta}\cap \Omega_{\epsilon}\right)\leq r\}$  forces a kind of interior six-arm event to appear, whose probability can be bounded using the corresponding arm exponent, which is strictly bigger than $2$.  We first collect some properties of critical random-cluster models with $q\in[1,4)$: a strong RSW estimates (Lemma~\ref{lem::RSW}) and a universal interior five-arm exponent (Lemma~\ref{lem::5arm}). 

For a discrete quad $(G;a,b,c,d)$, we denote by $L=L(G; a, b, c, d)$ the discrete extremal distance between $(ab)$ and $(cd)$ in $G$; see~\cite[Section~6]{ChelkakRobustComplexAnalysis}. The discrete extremal distance is uniformly comparable to its continuous counterpart--the classical extremal distance. 
	We denote  by $\{(ab)\longleftrightarrow(cd)\}$ the event that there exists a path of open edges in random-cluster model on $G$ connecting $(ab)$ to $(cd)$ in $G$. 
	We will use the following strong RSW estimate.

	\begin{lemma}{\cite[Theorem~1.2]{DCMTRCMFractalProperties}} \label{lem::RSW}
		Let $q\in [1,4)$. 	For each $L_0>0$, there exists $u(L_0)>0$ such that the following holds: for any discrete quad $(G;a,b,c,d)$ with $L(G; a, b, c, d)\le L_0$ and any boundary condition $\xi$, we have
		\begin{equation*}
			\mu_{p_c(q), q,G}^{\xi} \left[(ab)\longleftrightarrow(cd)\right]\geq u(L_0).
		\end{equation*}
	\end{lemma}

We now introduce arm events. For $z\in \mathbb{C}$ and $0<r<R$, we write $\ball(z;r)=\{w\in \mathbb{C}: |w-z|<r\}$ and $A(z;r,R)=\{w\in \mathbb{C}:r<|z-w|<R\}$.  Suppose that $A(z;r,R)$ is inside of $\Omega^{\delta}$, we call a path of open edges (pattern 1) or a path of open dual edges (pattern 0) in the random-cluster model connecting $\partial \ball(z;r)$ to $\partial \ball(z;R)$ an \text{arm}. We denote by $\mathcal{A}_5(z;r,R)$ the event that there are five disjoint arms in $A(z;r,R)$ with patterns $1/0/1/0/1$ in counterclockwise order and denote by $\mathcal{A}_6(z;r,R)$ the event that there are six disjoint arms in $A(z;r,R)$ with patterns $1/0/1/1/0/1$ in counterclockwise order. 

\begin{lemma}  \label{lem::5arm}
	Let $q\in [1,4)$.  Let $A(z;r,R)\subseteq \Omega$ be an annulus such that 
	\[ \Dist\left(A(z;r,R),\partial\Omega\right)\geq \frac{\epsilon}{5} \quad\text{and}\quad R\leq\frac{\epsilon}{20}.\]
	 Then there exist two constants $c_3,c_4\in (0,\infty)$ which depend only on $q $ and $\epsilon$ such that for any boundary condition $\xi$, when $\delta$ is small enough, we have
	\begin{equation} \label{eqn::five_arm_proba}
		c_3 \left(\frac{r}{R}\right)^2\leq \mu^{\xi}_{p_c(q),q,\Omega^{\delta}} \left[\mathcal{A}_5(z;r,R)\right]\leq c_4  \left(\frac{r}{R}\right)^2. 
	\end{equation}
\end{lemma}
\begin{proof}
		When $\mu_{p_c(q),q,\Omega^{\delta}}^{\xi}$ is replaced by the infinite-volume measure $\mu_{p_c(q),q,\delta\mathbb{Z}^2}^{\xi}$, the estimate~\eqref{eqn::five_arm_proba} is given in~\cite[Proposition~6.6]{DCMTRCMFractalProperties}. Combining this with the mixing property in~\cite[Eq.~(2.4)]{DCMTRCMFractalProperties}, we obtain~\eqref{eqn::five_arm_proba} with the measure $\mu_{p_c(q),q,\Omega^{\delta}}^{\xi}$.
\end{proof}

With Lemmas~\ref{lem::RSW}-\ref{lem::5arm} at hand, the following estimate on the probability of $\mathcal{A}_6(z;r,R)$ follows from a standard argument (see e.g.,~\cite[Proof of Corollary~6.7]{DCMTRCMFractalProperties}). 
 \begin{lemma}\label{lem::six_arm}
 	Let $q\in [1,4)$.  Let $A(z;r,R)\subseteq \Omega$ be an annulus such that 
	\[ \Dist\left(A(z;r,R),\partial\Omega\right)\geq \frac{\epsilon}{5} \quad\text{and}\quad R\leq\frac{\epsilon}{20}.\]
 	Then there exist constants $c_5,c_6\in (0,\infty)$ that depend only on $q $ and $\epsilon$ such that for any boundary condition $\xi$, when $\delta$ is small enough, we have
 	\begin{equation} \label{eqn::six_arm}
 		\mu^{\xi}_{p_c(q),q,\Omega^{\delta}} \left[\mathcal{A}_6(z;r,R)\right]\leq c_5 \left(\frac{r}{R}\right)^{2+c_6}. 
 	\end{equation}
 \end{lemma}
For completeness, we provide a proof below. To this end, we introduce two auxiliary events $\hat{\mathcal{A}}_6(z;r,R)$ and $\hat{\mathcal{A}}_5(z;r,R)$, as follows. Suppose that $\delta$ is small enough. We define $\hat{\mathcal{A}}_6(z;r,R)\subseteq \mathcal{A}_6(z;r,R)$ to be the following event: for $j=1,3,5$ (resp, $j=2,4,6$), there exists a path of open edges (resp., open dual edges) in 
 	\[\hat{A}:=A(z;r,R)\cup \Big\{z'=r'e^{\ii\theta}: r'\in [r-2\delta,r+2\delta]\cup [R-2\delta,R+2\delta],\, \theta\in [j\pi/3-\pi/9, j\pi/3+\pi/9]\Big\}\]
 	connecting $\big\{z'=re^{\ii\theta}: \theta\in [j\pi/3-\pi/9, j\pi/3+\pi/9]\big\}$ to $\big\{z'=Re^{\ii\theta}: \theta\in [j\pi/3-\pi/9, j\pi/3+\pi/9]\big\}$. We also define $\hat{\mathcal{A}}_5(z;r,R)\subseteq \mathcal{A}_5(z;r,R)$ to be the following event: (1) for $j=1,3,5$ (resp, $j=2,4$), there exists a path of open edges (resp., open dual edges) in $\hat{A}$
 	connecting $\big\{z'=re^{\ii\theta}: \theta\in [j\pi/3-\pi/9, j\pi/3+\pi/9]\big\}$ to $\big\{z'=Re^{\ii\theta}: \theta\in [j\pi/3-\pi/9, j\pi/3+\pi/9]\big\}$; (2) these $5$ paths are disjoint. As a consequence of~\cite[Proposition~6.5]{DCMTRCMFractalProperties} and the mixing property in~\cite[Eq. (2.4)]{DCMTRCMFractalProperties}, under the setup of Lemma~\ref{lem::six_arm},
  there exists a constant $c_* \in(0,\infty)$ that depends only on $q$ and $\epsilon$ such that 
 	\begin{align}\label{eqn::six_arm_proba_aux_1}
 		\mu^{\xi}_{p_c(q),q,\Omega^\delta}[\hat{\mathcal{A}}_6(z;r,R)]\geq c_* \mu^{\xi}_{p_c(q),q,\Omega^\delta}[\mathcal{A}_6(z;r,R)].
 	\end{align}
 	\begin{proof}[Proof of Lemma~\ref{lem::six_arm}]
 		Note that $\hat{\mathcal{A}}_6(z;r,R)\subseteq \hat{\mathcal{A}}_5(z;r,R)$. On the event $ \hat{\mathcal{A}}_5(z;r,R)$, we explore the states of edges in $\hat{A}\cap \delta\mathbb{Z}^2$, as follows. First, we explore the open clusters that are connected to $\{z'=re^{\ii\theta}:\theta\in [8\pi/9, 10\pi/9]\}$ and the dual open clusters that are connected to $\{z'=re^{\ii\theta}:\theta\in [5\pi/9, 7\pi/9]\cup [11\pi/9, 13\pi/9]\}$. 
For $j\in \{2,4\}$, fix an arbitrary path of dual open edges $\gamma_j$ connecting $\big\{z'=re^{\ii\theta}: \theta\in [j\pi/3-\pi/9, j\pi/3+\pi/9]\big\}$ to $\big\{z'=Re^{\ii\theta}: \theta\in [j\pi/3-\pi/9, j\pi/3+\pi/9]\big\}$. 
 		Second, we explore the path of open edges $\gamma_1$ connecting $\{z'=re^{\ii\theta}: \theta\in [2\pi/9, 4\pi/9]\}$ to $\{z'=Re^{\ii\theta}: \theta\in [2\pi/9, 4\pi/9]\}$ that is nearest to $\gamma_2$, and explore the path of open edges $\gamma_5$ connecting $\{z'=re^{\ii\theta}: \theta\in [14\pi/9, 16\pi/9]\}$ to $\{z'=Re^{\ii\theta}: \theta\in [14\pi/9, 16\pi/9]\}$ that is nearest to $\gamma_4$. 
 		
 		On the event $\hat{\mathcal{A}}_{5}(z;r,R)$, the two paths $\gamma_1$ and $\gamma_5$ are disjoint. We denote by $\hat{\Omega}$ the simply connected domain $\hat{\Omega}$ surrounded by $\gamma_1\cup \gamma_5\cup \partial A(z;r,R)$. On the event $\hat{\mathcal{A}}_6(z;r,R)$, there exists a path of dual open edges in $\hat{\Omega}$ connecting $\partial D(z;r)$ to $\partial D(z;R)$. It then follows from Lemma~\ref{lem::RSW} that there exist two constants $c_6,\hat{c}\in (0,\infty)$ that depends only on $q$ such that 
 		\begin{align} \label{eqn::six_arm_proba_aux_2}
 	\mu^{\xi}_{p_c(q),q,\Omega^{\delta}}[\hat{\mathcal{A}}_6(z;r,R)\vert \hat{\mathcal{A}}_5(z;r,R)]\leq \hat{c} \left(\frac{r}{R}\right)^{c_6}.
 		\end{align}
 		Combining~\eqref{eqn::six_arm_proba_aux_2} with~\eqref{eqn::six_arm_proba_aux_1} and Lemma~\ref{lem::5arm}, we have
 		\begin{align*}
 			\mu^{\xi}_{p_c(q),q,\Omega^{\delta}}[\mathcal{A}_6(z;r,R)]\leq \frac{1}{c_*} \mu^{\xi}_{p_c(q),q,\Omega^{\delta}}[\hat{\mathcal{A}}_6(z;r,R)]\leq \frac{\hat{c}c_4}{c_*} \left(\frac{r}{R}\right)^{2+c_6},
 		\end{align*}
 		as we set out to prove. 
 	\end{proof}

Now, we are ready to prove Lemma~\ref{lem::FK_mono_aux2}. 
\begin{proof}[Proof of Lemma~\ref{lem::FK_mono_aux2}]
	Thanks to RSW estimates in Lemma~\ref{lem::RSW}, it suffices to prove~\eqref{eqn::FK_mono_aux4} with $\mathbb{Q}^{\delta}$ replaced by $\mathbb{P}^{\delta}$. We find $z_1,\ldots, z_M\in \Omega_{\epsilon}$ such that $\Omega_{\epsilon}\subseteq \cup_{j=1}^M \ball(z_j;r/2)$. Note that for all $r\in (0,\epsilon/100)$, we can choose those $z_j$ such that $M\leq c_7/r^2$, where $c_7$ is a constant that depends only on $\Omega$. 
	
	We write
	\[E^{\delta}:=\vartheta^{(2)}(\Omega^{\delta};x_1^{\delta},x_2^{\delta},x_3^{\delta},x_4^{\delta})\cap \{\Dist\left(\eta_1^{\delta}\cap \Omega_{\epsilon}, \eta_2^{\delta}\cap \Omega_{\epsilon}\right)< r\}.\]
	Note that on the event $E^{\delta}$, there exists a vertex $x^{\delta}\in \Omega^{\delta}\cap \Omega_{\epsilon}$ such that $\mathcal{A}_6(x^{\delta};r/3,\epsilon/3)$ happens. Let $k\in \{1,2,\ldots, M\}$ be such that $x^{\delta}\in \ball(z_{k};r/2)$. Then the event $\mathcal{A}_6(z_{k};r,\epsilon/20)$ happens. From~\eqref{eqn::six_arm}, we have
	\begin{align*}
		\mathbb{P}^{\delta}\left[ E^{\delta}\right] \leq \mathbb{P}^{\delta}\left[\cup_{j=1}^M \mathcal{A}_6(z_j;r,\epsilon/5)\right]
		\leq \sum_{j=1}^M  \mathbb{P}^{\delta}\left[\mathcal{A}_6(z_j;r,\epsilon/5)\right]
		\leq  20^{2+c_6}c_5c_7\frac{r^{c_6}}{\epsilon^{2+c_6}}. 
	\end{align*}
 We obtain~\eqref{eqn::FK_mono_aux4} with $\mathbb{Q}^{\delta}$ replaced by $\mathbb{P}^{\delta}$, as desired.
\end{proof}

\appendix

\section{Relation between cross-ratio and extremal distance}
\label{appendix::crossratio_extremaldistance}

\begin{lemma}\label{lem::Ls}
Fix a quad $(\Omega; x_1, x_2, x_3, x_4)$. We consider the following two conformally invariant quantities of the quad. 
\begin{itemize}
\item Cross-ratio. Let $\varphi$ be a conformal map from $\Omega$ onto $\HH$ such that $\varphi(x_1)<\varphi(x_2)<\varphi(x_3)<\varphi(x_4)$. We define the cross-ratio of the quad $(\Omega; x_1, x_2, x_3, x_4)$ by
\begin{equation*}
s=s(\Omega; x_1, x_2, x_3, x_4)=\frac{(\varphi(x_4)-\varphi(x_1))(\varphi(x_3)-\varphi(x_2))}{(\varphi(x_3)-\varphi(x_1))(\varphi(x_4)-\varphi(x_2))}. 
\end{equation*} 
\item Extremal distance. 
There exists a unique $L=L(\Omega; x_1, x_2, x_3, x_4)>0$ and a unique conformal map $\phi$ from $\Omega$ onto the rectangle $[0,1]\times[0,\ii L/\pi]$ which sends $(x_1, x_2, x_3, x_4)$ to the four corners $(\ii L/\pi, 0, 1, 1+\ii L/\pi)$. We call $L=L(\Omega; x_1, x_2, x_3, x_4)$ the extremal distance between $(x_2x_3)$ and $(x_4x_1)$ in $\Omega$. 
\end{itemize}
We write the cross-ratio $s=s(L)$ as a function in $L$.  
Then the following limit exists 
	\begin{equation}\label{eqn::asyLs}
		\lim_{L\to \infty} s(L)\times\exp(L)= C\in (0,\infty). 
\end{equation}
\end{lemma}
\begin{proof}
From conformal invariance, we may assume $(\Omega; x_1, x_2, x_3, x_4)=(\HH; \infty, 0, s, 1)$. Schwarz-Christoffel formula gives
\begin{equation*}
\phi(z)=\frac{\int_0^z (u(u-s)(u-1))^{-1/2}\ud u}{\int_0^s (u(u-s)(u-1))^{-1/2}\ud u},\quad z\in\HH. 
\end{equation*}
In particular, $\phi(1)=1+\ii L/\pi$ gives 
\begin{equation}\label{eqn::Ls}
\frac{L}{\pi}=\frac{\int_s^1 (u(u-s)(1-u))^{-1/2}\ud u}{\int_0^s (u(s-u)(1-u))^{-1/2}\ud u}. 
\end{equation}
Note that $s\to 0$ as $L\to\infty$ and our goal is to derive the convergence~\eqref{eqn::asyLs}. 

For the denominator, we have
\begin{align*}
\int_0^s (u(s-u)(1-u))^{-1/2}\ud u=&\int_0^1 (t(1-t)(1-st))^{-1/2} \ud t \tag{set $t=u/s$}\\
=&\int_0^1 \left((t(1-t)(1-st))^{-1/2} -(t(1-t))^{-1/2}\right)\ud t+\int_0^1 (t(1-t))^{-1/2}\ud t  \notag\\
=&\int_0^1 \frac{st}{(t(1-t)(1-st))^{1/2}(1+(1-st)^{1/2})} \ud t+\pi. 
\end{align*}
Thus, 
\begin{align}
\int_0^s (u(s-u)(1-u))^{-1/2}\ud u=&O(s)+\pi. \label{eqn::Ls:aux1}
\end{align}

For the numerator, we write
\begin{equation*}\label{eqn::Ls:aux2}
	\int_s^1 (u(u-s)(1-u))^{-1/2}\ud u= \underbrace{
		\int_{1/2}^1 (u(u-s)(1-u))^{-1/2}\ud u}_{Z_1(s)}+\underbrace{\int_s^{1/2} (u(u-s)(1-u))^{-1/2}\ud u}_{Z_2(s)}.
\end{equation*}
We treat the terms $Z_1(s)$ and $Z_2(s)$ one by one.

\begin{itemize}
		\item For $Z_1(s)$, by dominated convergence theorem, we have 
	\begin{equation} \label{eqn::Ls:aux4}
		\lim_{s\to 0} Z_1(s)= \int_{1/2}^1 u^{-1}(1-u)^{-1/2}\ud u.
	\end{equation}
	\item For $Z_2(s)$, we write
	\begin{align} 
		Z_2(s)+\log s =&\int^{1/2}_s (u(u-s)(1-u))^{-1/2}\ud u-\int_s^{1/2}u^{-1} \ud u-\log 2\notag\\
		 =&\underbrace{\int_s^{1/2} u^{-1/2}\left((u-s)^{-1/2}-u^{-1/2}\right)\ud u}_{Z_3(s)}+\underbrace{\int_s^{1/2} (u(u-s))^{-1/2}((1-u)^{-1/2}-1)\ud u}_{Z_4(s)}-\log 2.\label{eqn::l_s:aux31}
	\end{align}
	For $Z_3(s)$, by change of variables $t=u/s$, we have 
	\begin{equation}\label{eqn::l_s:aux33}
		Z_3(s)=\int^{1/(2s)}_1 t^{-1/2}\left((t-1)^{-1/2}-t^{-1/2}\right) \ud t\to  \int_1^{\infty}t^{-1/2}\left((t-1)^{-1/2}-t^{-1/2}\right)  \ud t,\quad \textit{as }s\to 0. 
	\end{equation}
For $Z_4(s)$, by change of variables $t=u-s$, we have 
\begin{equation}\label{eqn::l_s:aux311}
	Z_4(s)=\int_0^{1/2} ((t+s)t)^{-1/2} \left((1-t-s)^{-1/2}-1\right) \mathbb{1}_{\{t\leq 1/2-s\}}\ud t.
\end{equation}
For the integrand in~\eqref{eqn::l_s:aux311}, by mean value theorem, we have 
\begin{equation*} 
((t+s)t)^{-1/2} \left((1-t-s)^{-1/2}-1\right) \mathbb{1}_{\{t\leq 1/2-s\}} \leq 2^{-5/2} (t+s)^{1/2}t^{-1/2} \mathbb{1}_{\{t\leq 1/2-s\}} \leq 2^{-3/2} t^{-1/2},\quad \forall t\in (0,1/2).
\end{equation*}
Thus, for $Z_4(s)$, by dominated convergence theorem, we have 
\begin{equation}\label{eqn::l_s:aux32}
	Z_4(s)\to  \int_0^{1/2}t^{-1} \left((1-t)^{-1/2}-1\right)\ud t,\quad \textit{as }s\to 0.
\end{equation}

Plugging~\eqref{eqn::l_s:aux33} and~\eqref{eqn::l_s:aux32} into~\eqref{eqn::l_s:aux31}, we obtain
\begin{equation} \label{eqn::Ls:aux3}
\lim_{s\to 0} \left(Z_2(s)+\log s\right)=c\in \mathbb{R}, 
\end{equation}
where 
\begin{equation}\label{eqn::l_s:aux34}
	c:=\int_0^{1/2}t^{-1} \left((1-t)^{-1/2}-1\right)\ud t+\int_1^{\infty}t^{-1/2}\left((t-1)^{-1/2}-t^{-1/2}\right)  \ud t-\log 2.
\end{equation}

\end{itemize}

Plugging~\eqref{eqn::Ls:aux1},~\eqref{eqn::Ls:aux4} and~\eqref{eqn::Ls:aux3} into~\eqref{eqn::Ls},  we have 
\begin{align*}
	\frac{L+\log s}{\pi}
	&=\frac{\int_s^1 (u(u-s)(1-u))^{-1/2}\ud u}{\int_0^s (u(s-u)(1-u))^{-1/2}\ud u}+\frac{\log s}{\pi}\\
	&=\frac{\int_s^1 (u(u-s)(1-u))^{-1/2}\ud u+\log s}{\int_0^s (u(s-u)(1-u))^{-1/2}\ud u}+\frac{O(s)\log s }{\int_0^s (u(s-u)(1-u))^{-1/2}\ud u}\\
	&\to \frac{c+\int_{1/2}^1 u^{-1}(1-u)^{-1/2}\ud u}{\pi},\quad \textit{as }s\to 0.
\end{align*}
where $c$ is given by~\eqref{eqn::l_s:aux34}.
Thus, we obtain~\eqref{eqn::asyLs} with
\begin{equation*}
	C=\exp\left(c+\int_{1/2}^1 u^{-1}(1-u)^{-1/2}\ud u\right).
	\end{equation*}
\end{proof}


{\small \bibliographystyle{alpha}
\bibliography{bibliography}}

\end{document}